\documentclass[10pt,a4paper]{article}
\usepackage{cite}
\usepackage{bm}
\usepackage{amsmath}
\makeatletter
\def\maketag@@@#1{\hbox{\m@th\normalfont\normalsize#1}}
\makeatother
\usepackage{amssymb,amsfonts}

\usepackage{mathtools}
\usepackage{algorithm}
\usepackage{algcompatible}
\usepackage{graphicx}
\usepackage{textcomp}
\usepackage[hidelinks]{hyperref}
\usepackage{subcaption}
\usepackage{pgfplots}
\usepackage{tikz}
\usetikzlibrary{calc,patterns,
	decorations.pathmorphing,
	decorations.markings}
\pgfplotsset{compat=newest}
\newcommand{\argmin}{\mathop{\rm arg\ min}\nolimits}
\newcommand{\blkdiag}{\mathop{\rm blockdiag}\nolimits}
\newcommand{\Jac}{\bm{\text{J}}}
\newcommand{\la}{linear algebra }
\newcommand{\nz}{nonzero }
\newcommand{\nzs}{nonzeros }
\usepackage[bottom]{footmisc}
\newcommand\blfootnote[1]{%
	\begingroup
	\renewcommand\thefootnote{}\footnote{#1}%
	\addtocounter{footnote}{-1}%
	\endgroup
}

\newenvironment{proof}{\textbf{Proof:}}{\hfill$\blacksquare$\\}
\newtheorem{proposition}{Proposition}
\newtheorem{remark}{Remark}
\newtheorem{theorem}{Theorem}
\newtheorem{lemma}{Lemma}
\newtheorem{corollary}{Corollary}
\newtheorem{definition}{Definition}
\algnewcommand\INPUT{\item[\textbf{Inputs:}]}%
\algnewcommand\OUTPUT{\item[\textbf{Outputs:}]}%
\newcommand{\algrule}[1][.2pt]{\par\vskip.5\baselineskip\hrule height #1\par\vskip.5\baselineskip}
\title{An efficient bounded-variable nonlinear least-squares algorithm for embedded MPC}

\author{
	Nilay~Saraf ~and~
	Alberto~Bemporad
}
\begin{document}

\maketitle
\blfootnote{This project had received funding from the European Union's Horizon 2020 Framework Programme for Research and Innovation under grant agreement No.\ 674875 (oCPS). 
	
N.\ Saraf (e-mail: nilay.saraf@alumni.imtlucca.it) and

A.\ Bemporad (e-mail: alberto.bemporad@imtlucca.it) are with the IMT School for Advanced Studies Lucca, Piazza San Francesco 19, Lucca, 55100 LU Italy .}	
\begin{abstract}
	This paper presents a new approach to solve linear and nonlinear model predictive control (MPC) problems that requires small memory footprint and throughput and is particularly suitable
	when the model and/or controller parameters change at runtime. 
	Typically MPC requires two phases: 1) construct an optimization
	problem based on the given MPC parameters (prediction model, tuning weights,  prediction horizon, and constraints), which results in a quadratic or nonlinear programming problem, and then 2) call an optimization algorithm to solve the resulting problem. In the proposed approach the problem construction step is systematically eliminated, as in the optimization algorithm problem matrices are expressed in terms of abstract functions of the MPC parameters. 
	We present a unifying algorithmic framework based on active-set
	methods with bounded variables that can cope with linear, nonlinear, and adaptive MPC variants based on a broad class of prediction models and a sum-of-squares cost function. The theoretical and numerical results demonstrate the potential, applicability, and efficiency of the proposed framework for practical real-time embedded MPC.
\end{abstract}

\noindent \textbf{Keywords: Model predictive control, active-set methods, nonlinear parameter-varying control,  adaptive control, nonlinear programming, sparse recursive QR factorization.}
\section{Introduction}
\label{sec:introduction}
Model predictive control has evolved over the years from a method developed for controlling slow processes~\cite{RS68,mpcApps} to an advanced multivariable control method that is applicable even to fast-sampling applications, such as in the automotive and aerospace domains~\cite{BBLV18,AnAapplications}. This evolution has been
possible because of the significant amount of research on computationally efficient real-time MPC algorithms. For an incomplete list of such efforts and tools the reader is referred to~\cite{effNMPCdiehl,qpnmpcdiehl,panoc,cannon,ohtsuka,RTI,fastMPConlineOpt,MHzMPC}.  Despite the success of MPC, demand for faster numerical algorithms for a wider scope of applications has been reported for instance in~\cite{AnAapplications}. A common approach to reduce computational load is to solve the MPC problem suboptimally, see for instance~\cite{RTI,fastMPConlineOpt}. However, even such MPC approaches have limitations that could be prohibitive in some resource-constrained applications, especially in the case of (parameter-varying) nonlinear MPC (NMPC). This denotes that there is still a large scope of improvement. 

A usual practice in MPC is to first formulate an optimization problem based on the prediction model and MPC tuning parameters, before passing it in a standard form to an optimization solver. 
Such a problem construction step can be performed \textit{offline} when the prediction model is time-invariant, such as linear time-invariant (LTI)  model of the system, whereas it needs to be repeated at each instance in case of parameter-varying models, such as nonlinear models linearized at the current operating point, or changes of MPC tuning parameters (such as prediction horizon, control horizon, tuning weights, or sampling time). 
Often, constructing the optimization problem requires a computation effort comparable to that required for solving the optimization problem itself. 
The same occurs in the recently proposed data-driven MPC scheme~\cite{ddmpc} where due to potentially time-varying model and/or tuning parameters, re-constructing the MPC optimization problem on line becomes necessary, which significantly increases the computational load. Notwithstanding these limitations of MPC, scarcely any effort has been made till date to design a real-time MPC approach which does not need (re-)construction of the optimization problem with varying model and/or MPC tuning parameters. Approaches which partially address this aspect for a limited class of linear parameter-varying (LPV) models with a fixed MPC problem structure include~\cite{rebuildFreeLPVMPC,fastMPConlineOpt}.

The methods proposed in this paper aim at reducing the computational complexity of MPC while eliminating the optimization problem construction step even for the general case of nonlinear parameter-varying systems, through algorithms that can adapt to changes in the model and/or tuning parameters at runtime. The main ideas employed for this purpose are: 1)
a structured and sparse formulation of the MPC problem through a quadratic penalty function in order to exploit simple and fast solution methods, 2) replacing matrix instances via abstract operators that map the model and tuning parameters to the result of the required matrix operations in the optimization algorithm. Besides this, the contributions of this paper include: 1) an overview on the relation between quadratic penalty and augmented Lagrangian methods w.r.t.\ their application for equality constraint elimination in the considered (nonlinear) MPC problems, 2) a discussion on alternative methods to implement MPC with resulting solution algorithms having negligible increase in memory requirement w.r.t.\ the number of decision variables, 3) methods to exploit problem
sparsity and efficiently implement the active-set method proposed in~\cite{bvlsTN} for MPC based on box-constrained (nonlinear) least-squares. 

Regarding the last contribution, we note that each iteration of a primal active-set method~\cite{Nocedal} involves the solution of a linear system, which in the case of the algorithm in~\cite{bvlsTN} is a sparse unconstrained linear least-squares (LS) problem. These LS problems between successive iterations are related by a rank-one update. In~\cite{bvlsTN}, it has been shown for the numerically \textit{dense} case that, as compared to solving an LS problem from scratch, employing a recursive QR factorization scheme that exploits the relation between successive LS problems can significantly increase computational speed without compromising numerical robustness, even without using advanced \la libraries. In the \textit{sparse} case, even though very efficient approaches exist for solving a single LS problem using direct~\cite{colamd} or iterative~\cite{SparseMats} methods with sparse linear algebra libraries, to the best of the authors' knowledge no methods have been reported for recursively updating the sparse QR factorization of a matrix. A recursive approach for sparse LU factorization has been described in~\cite{recLU}; however, such an approach not only needs the storage of the matrix and its sparsity pattern, which requires constructing the MPC problem and forming the normal equations that could be numerically susceptible to ill-conditioning, but it also relies on \la packages that could be cumbersome to code, especially in an embedded control platform. In this paper, we present novel methods for numerically stable sparse recursive QR factorization based on Gram-Schmidt orthogonalization, which are easy to implement and are very efficient even for small-size problems, therefore extending the dense approach of~\cite{bvlsTN}. Although the proposed methods are designed for the specific MPC application, i.e., to solve the sparse LS problems having a specific parameter-dependent structure without forming the matrix that is factorized, they may be applicable for other LS problems with block-sparse matrices having similar special structures. 

The paper is organized as follows. Section~\ref{sec:prelim} describes the considered general class of discrete-time models and MPC problem formulation. The proposed non-condensed formulation based on eliminating the equality constraints due to the model equations is motivated and described in detail in Section~\ref{sec:elimeq}. In Section~\ref{sec:solvers} we describe a solution algorithm for bound-constrained nonlinear least-squares optimization with a theoretical analysis on its global convergence. A parameterized implementation of this algorithm for solving MPC problems without the construction phase and relying on the abstraction of matrix instances is described in Section~\ref{sec:abstractmat}. Methods for sparse recursive thin QR factorization are described in Section~\ref{sec:sparseQR}. Section~\ref{sec:results} briefly reports numerical results based on a nonlinear MPC (NMPC) benchmark example
that clearly demonstrate the very good computational performance of the proposed methods against other methods. Finally,
Section~\ref{sec:conclusions} concludes the paper.

Excerpts of Sections~\ref{sec:prelim}-\ref{sec:elimeq}, and Section~\ref{sec:bvnlls} are based on the authors' conference papers~\cite{bvlsmpc,nmpcbvls}. These papers introduced the idea to formulate the (N)MPC problem using a quadratic penalty function for a fast solution using bounded-variable (nonlinear) least squares with numerically dense computations. All the other ideas that we have introduced above are proposed in this paper and are original. In addition, we include corrections and extensions of the contents in common with~\cite{bvlsmpc,nmpcbvls}. 
\subsection*{Basic notation} \noindent
We denote the set of real vectors of dimension $n$ as $\mathbb{R}^n$; a real matrix with $m$ rows and $n$ columns as $A\in\mathbb{R}^{m\times n}$; its transpose as $A^{\top}$, its inverse as $A^{-1}$, and its pseudo-inverse as $A^{\dagger}$. For a vector $a\in\mathbb{R}^{m}$, its $p$-norm is $\Vert a\Vert_p$, its $j{\text{th}}$ element is $a(j)$, and $\Vert a\Vert_2^2=a^\top a$. A vector or matrix with all zero elements is represented by $\bm{0}$. If $\mathcal{F}$ denotes a set of indices, $A_{\mathcal{F}}$ denotes a matrix formed from columns of $A$ corresponding to the indices in $\mathcal{F}$. Given $N$ square matrices $A_1,\ldots,A_N$, of possible different orders, $\blkdiag(A_1,\ldots,A_N)$ is the block diagonal matrix
whose diagonal blocks are $A_1,\ldots,A_N$.

For scalars $a$ and $b$, $\min(a, b)$ and $\max(a, b)$ denote, respectively, the minimum and maximum of the two values. 
Depending on the context, $(a, b]$ or $[b,a)$ represent either the set of real numbers or integers between $a$ and $b$, excluding $a$ and including $b$, or vice-versa.

The gradient of a function $f:\mathbb{R}^n\to\mathbb{R}$ at a point $\bar x\in\mathbb{R}^n$ is either denoted by $\left.\nabla_x f(x)\right\vert_{\bar x}$ or $\nabla _x f(\bar x)$, the Hessian matrix by $\nabla_x^2 f(\bar x)$; the Jacobian 
of a vector function $g:\mathbb{R}^n\to\mathbb{R}^m$
by  $\left.\Jac_x g(x)\right\vert_{\bar x}$ or $\Jac g(\bar x)$. 

Finite sets of elements are represented 
by curly braces containing the elements; $\emptyset$ denotes the empty set. If a set $\mathcal A$ is a subset of set $\mathcal B$ (i.e., if $\mathcal B$ is the superset of $\mathcal A$), it is written as $\mathcal{A}\subseteq \mathcal{B}$ (or alternatively $\mathcal{B}\supseteq \mathcal{A}$). The symbols $\cup, \cap$, and $\backslash$ between two sets denote, respectively, set union, intersection, and difference. The summation notation for sets is denoted by $\bigcup$. 
The cardinality (number of elements) of a finite set $\mathcal{A}$ is denoted by $\vert \mathcal{A} \vert$.

\section{Preliminaries}\label{sec:prelim}
For maximum generality, the prediction model we use in MPC is described by the following discrete-time multivariable nonlinear parameter-varying dynamical model equation
\begin{equation} \label{eq:model}
	\mathcal{M}(Y_k,~U_k,~S_k) = \bm{0},
\end{equation}
where $U_k = (u_{k-n_{\mathrm{b}}},\ldots,u_{k-1})$ with $u_k\in \mathbb{R}^{n_u}$
the input vector at sampled time step $k$, and $Y_k = (y_{k-n_{\mathrm{a}}},\ldots,y_k)$ with  $y_k\in \mathbb{R}^{n_y}$ the output vector at $k$. Vector $S_k = (s_{k-n_{\mathrm{c}}},\ldots,s_{k-1})$, where $s_k\in \mathbb{R}^{n_s}$, $n_s\geq 0$, contains possible exogenous signals, such as measured disturbances. 

We assume that function $\mathcal{M}:\mathbb{R}^{n_{\mathrm{a}} n_y} \times \mathbb{R}^{n_{\mathrm{b}} n_u} \times \mathbb{R}^{n_{\mathrm{c}} n_s} \to \mathbb{R}^{n_y}$ is differentiable, where $n_{\mathrm{a}}, n_{\mathrm{b}}$ and $n_{\mathrm{c}}$ denote the model order. Special cases include deterministic nonlinear parameter-varying auto-regressive exogenous (NLPV-ARX) models, state-space models ($y$ = state vector, $n_a=n_b=n_c=1$), neural networks with a smooth activation function, discretized first-principles models and differential algebraic equations. Designing the MPC controller based on the input-output (I/O) difference equation~\eqref{eq:model} has several benefits such as: 1) data-based black-box models which are often identified in I/O form do not need a state-space realization for control, 2) a state estimator is not required when all output and exogenous variables are measured, 3) the number of decision variables to be optimized does not scale
linearly with the number of system states but outputs, which could be fewer in number, 4) input delays can easily be incorporated in the model by simply shifting the sequence in $U_k$ backwards in time. 

\noindent Linearizing~\eqref{eq:model} w.r.t.\ a sequence of inputs $\hat{U}$ (that is, $U_k=\hat U+\Delta U$) and outputs $\hat{Y}$ ($Y_k=\hat Y+\Delta Y$) gives \par \vspace{-10pt}
\footnotesize
\begin{multline*} 
	\mathcal{M}(\hat{Y},~\hat{U},~S_k) + \left(\left.\Jac_{Y_k}\mathcal{M}(Y_k,~U_k,~S_k)\right|_{\hat{Y},~\hat{U}}\right) \Delta Y + \left(\left.\Jac_{U_k}\mathcal{M}(Y_k,~U_k,~S_k)\right|_{\hat{Y},~\hat{U}}\right) \Delta U= \bm{0},
\end{multline*}
\normalsize
which is equivalently written as the following affine parameter-varying I/O model, i.e., \par \vspace{-15pt}
\footnotesize
\begin{multline}\label{eq:arxmimovec}
	-A\left(S_k\right)_0 \Delta y_{k} = \sum_{j=1}^{n_{\mathrm{a}}} A\left(S_k\right)_j  \Delta y_{{k-j}} +\sum_{j=1}^{n_{\mathrm{b}}} B\left(S_k\right)_j  \Delta u_{{k-j}} + \mathcal{M}(\hat{Y},~\hat{U},~S_k),
\end{multline}\normalsize
where the Jacobian matrices 
\footnotesize
\begin{align*}
	A\left(S_k\right)_j &= \left.\Jac_{y_{k-j}} \mathcal{M}(Y_k,~U_k,~S_k)\right|_{\hat Y,~\hat U} ~\in \mathbb{R}^{n_y \times n_y},~\forall j \in [0, n_a], \\
	B\left(S_k\right)_j &= \left.\Jac_{u_{k-j}} \mathcal{M}(Y_k,~U_k,~S_k)\right|_{\hat Y,~\hat U} ~\in \mathbb{R}^{n_y \times n_u},~\forall j \in [1, n_b]. 
\end{align*}\normalsize

Note that for the special case of LTI models in ARX form, in~\eqref{eq:arxmimovec} $A_0$ would be an identity matrix whereas $S_k$ would be absent and $\mathcal{M}(\hat{Y},~\hat{U}) = \bm{0}$,
$\hat{Y}= \bm{0}$, $\hat{U}=\bm{0}$.

We consider the following performance index (P) which is commonly employed for reference tracking in MPC: \par \vspace{-14pt}
{\footnotesize
	\begin{multline} \label{eq:cost} 
		\hspace{-10pt} \text{P}_k=\sum\limits_{j=1}^{N_{\mathrm{p}}}\frac{1}{2}\Vert W_{y_{k+j}}(y_{k+j}-\bar{y}_{k+j})\Vert_2^2  + \sum\limits_{j=0}^{N_{\mathrm{u}}-2} \frac{1}{2}\Vert W_{u_{k+j}} (u_{k+j}-\bar{u}_{k+j})\Vert_2^2 \\+\frac{1}{2}(N_{\mathrm{p}}-N_{\mathrm{u}}+1)\cdot \Vert W_{u_{k+j}}(u_{k+N_{\mathrm{u}}-1}-\bar{u}_{k+N_{\mathrm{u}}-1})\Vert_2^2,
\end{multline}}\normalsize where $N_{\mathrm{p}}$ and $N_{\mathrm{u}}$ denote the prediction and control horizon respectively. Matrices $W_{y_{(\cdot)}} \in \mathbb{R}^{n_y\times n_y}$, $W_{u_{(\cdot)}} \in \mathbb{R}^{n_u\times n_u}$ denote tuning weights, 
vectors $\bar{y}$, $\bar{u}$ denote output and input references, respectively.

The methods described later in this paper can straightforwardly be extended to any performance index which is a sum of squares of linear or differentiable nonlinear functions.

The MPC optimization problem is formulated based on the cost function~\eqref{eq:cost} subject to constraints on the vector
$w_k=(u_k,\ldots,u_{k+N_{\mathrm{u}}-1},y_{k+1},\ldots,y_{k+N_{\mathrm{p}}})$ of input and output variables. In this paper we will consider equality constraints that arise from the prediction model~\eqref{eq:model}, and restrict inequality constraints to only simple bounds on 
input and output variables. General (\textit{soft}) inequality constraints~\eqref{eq:genineq} can nevertheless be included as equalities by introducing non-negative slack variables $\nu \in \mathbb{R}^{n_{\mathrm{i}}}$ such that
\begin{align}\label{eq:genineq}
	g(w_k) &\leq \bm{0} \quad \text{becomes},\\
	g(w_k) + \nu_k &= \bm{0} \quad \text{and}\notag \quad
	\nu_k \geq 0,
\end{align}
where $z_k=(w_k,\nu_k)$ and $g:\mathbb{R}^{(n_z-n_{\mathrm{i}})}\to \mathbb{R}^{n_{\mathrm{i}}}$ is assumed to be differentiable, while $n_z$ and $n_{\mathrm{i}}$ denote the number of decision variables and general inequality constraints respectively. In summary, the MPC optimization problem to be solved at each step $k$ is 
\begin{subequations}
	\label{eq:nlp}
	\begin{align}
		\min_{z_k} & \frac{1}{2}\Vert W_k(z_k-\bar{z}_k) \Vert_2^2\label{eq:nlpcost}\\
		\text{s.t. } & h_k(z_k,~\phi_k) = \bm{0}, \label{eq:eqineq}\\
		&p_k \leq z_k \leq q_k, \label{eq:box}
	\end{align}
\end{subequations}
where 
$p_k, q_k$ are vectors defining bounds on the input and output variables, and possible non-negativity constraint on slack variables. Some components of $z_k$ may be unbounded, in that case those bounds are passed to the solver we propose later as the largest negative or positive floating-point number in the computing platform, so that we can assume $p_k, q_k \in \mathbb{R}^{n_z}$. Vector $\phi_k=(u_{k-n_{\mathrm{b}}+1},\ldots,u_{k-1},y_{k},\ldots,y_{k-n_{\mathrm{a}}+1})$ denotes the initial condition whereas $\bar{z}$ contains references on the decision variables. The block-sparse matrix $W_k$ is constructed from the tuning weights ($W_{u_{(\cdot)}}, W_{y_{(\cdot)}}$) defined in~\eqref{eq:cost}. 

\section{Eliminating equality constraints}\label{sec:elimeq}
Handling equality constraints via penalty functions, or an augmented Lagrangian method, has proven to be effective for efficiently solving constrained optimization problems~\cite{Nocedal, BertsekasMM},\cite[Chapter 22]{LawsonHanson}. This section shows how similar methods can be applied to efficiently solve MPC problems of the form~\eqref{eq:nlp}. In order to employ fast solution methods, the general constrained problem~\eqref{eq:nlp} can be simplified as a box-constrained nonlinear least-squares (NLLS-box) problem by
using a quadratic penalty function and consequently eliminating the equality constraints~\eqref{eq:eqineq}  such that~\eqref{eq:nlp} becomes
\begin{align}\label{eq:nllsbox}
	\min_{p_k\leq z_k \leq q_k} & \frac{1}{2}\left\Vert \begin{array}{c} \frac{1}{\sqrt{\rho}} W_k(z_k-\bar{z}_k)\\ h_k(z_k,~\phi_k) \end{array}\right \Vert_2^2 \equiv\min_{p_k\leq z_k \leq q_k} \frac{1}{2}\Vert r_k(z_k)\Vert_2^2,
\end{align}
where the penalty parameter $\rho$ is a positive scalar and $r:\mathbb{R}^{n_z}\to \mathbb{R}^{n_r}$ denotes the vector of residuals. We propose the reformulation~\eqref{eq:nllsbox} of problem~\eqref{eq:nlp} for the following reasons:
\begin{enumerate}
	\item Penalizing the violation of equality constraints makes problem~\eqref{eq:nllsbox} always feasible;
	\item No additional slack variables are needed for softening output constraints, which would result in inequalities of general type instead of box constraints;
	\item While solving~\eqref{eq:nllsbox}, since we do not include additional slack variables to soften constraints, the function~$h_k$
	does not need to be analytic beyond bounds, which is discussed in further detail in Section~\ref{sec:solvers} (cf.\ Remark~\ref{rem:primfeas});
	\item No dual variables need to be optimized to handle equality constraints;
	\item Problem~\eqref{eq:nllsbox} is simpler to solve as compared to~\eqref{eq:nlp}, for instance, when using SQP algorithms (cf.\ Section~\ref{sec:solvers}), initializing a feasible guess is straightforward, the subproblems are feasible even with inconsistent linearizations of $h_k$, and convergence impeding phenomena such as the Maratos effect~\cite{Nocedal} are implicitly avoided.
\end{enumerate}
Relaxing the equality constraints as above also has an engineering justification~\cite{bvlsmpc}: As the prediction model~\eqref{eq:model} is only an approximation of the true system dynamics, (opportunistic) violations of the dynamic model equations will only affect the quality of predictions, depending on the magnitude of the violation.  Instead of using the iterative quadratic penalty method (QPM)~\cite[Framework 17.1]{Nocedal} with increasing values of $\rho$ in each iteration, we propose to use a single iteration with a large value of $\rho$ for solving~\eqref{eq:nlp}, owing to the fact that a good initial guess is often available in MPC. It has been proven in~\cite[Theorem 1]{bvlsmpc} that for a quadratic cost~\eqref{eq:nlpcost} subject to only consistent linear equality constraints, a single QPM iteration with sufficiently large penalty $\rho$ may result in negligible solution inaccuracy.
This has been clearly demonstrated by numerical examples in~\cite{bvlsmpc, nmpcbvls} for the general case. A practical upper bound on $\rho$ depends on the computing precision and numerical robustness of the optimization solver such that the Jacobian of the vector of residuals in~\eqref{eq:nllsbox} is numerically full-rank. 
The parameter $\rho$ is tuned (cf.\ \cite[Section 3.5.3]{thesisNS}) based on the fact that a higher value results in a lower solution inaccuracy at the cost of problem scaling which may affect the convergence rate of the adopted solution methods. A theoretical lower bound on $\rho$ exists and has been derived in~\cite{bvlsmpc} for the case of LTI systems based on closed-loop stability conditions. The extension of such a result to the general case is not immediate and thereby poses a risk given that the bound is not deterministic. However, in practice, based on the arguments and details in the references mentioned in this section, we expect a sufficiently low value of the equality constraint violation. 

An alternative approach to solve the optimization problem~\eqref{eq:nlp} without the equality constraints~\eqref{eq:eqineq} is the  bound-constrained Lagrangian method (BLM)~\cite[Algorithm 17.4]{Nocedal}, which can efficiently be solved by iteratively using the nonlinear gradient projection algorithm~\cite{Nocedal}. At each iteration $(i)$ of the BLM, one solves \par \vspace{-10pt}
\footnotesize
\begin{align}\label{eq:ALM}
	z_k^{(i+1)} =\arg \min\limits_{p_k\leq z_k \leq q_k}  \frac{1}{2}\left\Vert \begin{array}{c}  W_k(z_k-\bar{z}_k)\\ \sqrt{\rho^{(i)}} h_k(z_k,~\phi_k) \end{array}\right \Vert_2^2 + {\Lambda_k^\top}^{(i)} h_k(z_k,\phi_k)
\end{align}\normalsize
where $\Lambda$ denotes the vector of Lagrange multipliers corresponding to the equality constraints, and updates the estimates $\Lambda^{(i)}$ and $\rho^{(i)}$, until convergence (cf.\ \cite[Algorithm 17.4]{Nocedal}). 
\begin{proposition}\label{pro:alm_nllsbox}
	The optimization problem~\eqref{eq:ALM} is equivalent to the NLLS-box problem
	\begin{align}\label{eq:ALMbvnlls}
		z_k^{(i+1)} = \argmin_{p_k\leq z_k \leq q_k}  \frac{1}{2}\left\Vert \begin{array}{c}  \frac{1}{\sqrt{\rho^{(i)}}}W_k(z_k-\bar{z}_k)\\  h_k(z_k,~\phi_k) + \frac{\Lambda^{(i)}_k}{\rho^{(i)}} \end{array}\right \Vert_2^2
	\end{align}
\end{proposition}
\begin{proof} We have that problem 
	\begin{align*}\argmin_{p_k\leq z_k \leq q_k}  \frac{1}{2}\left\Vert \begin{array}{c}  W_k(z_k-\bar{z}_k) \\ \sqrt{\rho^{(i)}} h_k(z_k,~\phi_k) \end{array}\right \Vert_2^2&\\ + {\Lambda_k^\top}^{(i)} h_k&(z_k,\phi_k)\end{align*} 

	\begin{align*}\text{and~}\argmin_{p_k\leq z_k \leq q_k}  \frac{1}{2}\left\Vert W_k(z_k-\bar{z}_k) \right\Vert_2^2 + \mathcal{H}(z_k), \end{align*} where $\mathcal{H}(z_k)$ 
	\begin{align*} & = \frac{\rho^{(i)}}{2} \left\Vert h_k(z_k,~\phi_k) \right \Vert_2^2 + {\Lambda_k^\top}^{(i)} h_k(z_k,\phi_k) + \frac{\left\Vert\Lambda_k^{(i)}\right\Vert_2^2}{2\rho^{(i)}}\\ &= \frac{\rho^{(i)}}{2}\left(\left\Vert h_k(z_k,~\phi_k) \right \Vert_2^2 + \frac{2{\Lambda_k^\top}^{(i)} h_k(z_k,\phi_k)}{\rho^{(i)}} + \left\Vert\frac{\Lambda_k^{(i)}}{\rho^{(i)}}\right\Vert_2^2\right)\\ &=  \frac{\rho^{(i)}}{2} \left(h_k(z_k,~\phi_k) + \frac{\Lambda^{(i)}_k}{\rho^{(i)}}\right)^{\top}\left(h_k(z_k,~\phi_k) + \frac{\Lambda^{(i)}_k}{\rho^{(i)}}\right)\\ &= \frac{\rho^{(i)}}{2}\left\Vert h_k(z_k,~\phi_k) + \frac{\Lambda^{(i)}_k}{\rho^{(i)}}\right\Vert_2^2, \text{~are equivalent.}
	\end{align*} 
	Scaling by the constant $1/\rho^{(i)}$ yields the result.
\end{proof}	
\begin{remark}
	Proposition~\ref{pro:alm_nllsbox} holds for any sum-of-squares cost function with~\eqref{eq:nlpcost} as the special case, for instance $\Vert \mathcal{S}(z_k)\Vert_2^2$, where $\mathcal{S}$ is any vector-valued function. 
\end{remark}
Proposition~\ref{pro:alm_nllsbox} shows that we can employ the same NLLS-box solvers to solve~\eqref{eq:ALM}, which may be more efficient and numerically robust~(cf.\ Section~\ref{sec:solvers}) as compared to the use of other NLP solvers. When using BLM, sequences of $z_k^{(i)}$ and $\Lambda_k^{(i)}$ respectively converge to their optimal values $z_k^\star$ and $\Lambda^{\star}_k$ whereas $h_k(z_k^{\star},\phi_k) \approx \bm{0}$, numerically. Then via Proposition~\ref{pro:alm_nllsbox}, we note that for a fixed value of $\rho \gg \left\Vert\Lambda_k^{\star}\right\Vert_\infty$ in the equality-constrained case, we obtain\\ $h_k\left(z_k^{(i+1)},\phi_k\right)~\approx~\Lambda^{(i+1)}_k/\rho~\approx~\bm{0}$~\cite[Chapter 17]{Nocedal}, which is simply the solution obtained using a single iteration of QPM for the same $\rho$ and is consistent with the special case described by~\cite[Theorem 1]{bvlsmpc}. 

Although with BLM it is possible to initialize $\rho$ to aribtrarily low values and solve numerically easier problems, which is its main advantage over QPM, the final value of $\rho$ is not guaranteed to remain low. A main disadvantage of BLM over QPM is that it needs problem~\eqref{eq:nlp} to be feasible, otherwise the problem must be formulated with \textit{soft constraints} on output variables~\cite{Kerrigansoft}, which typically results in the use of penalty functions with large values of the penalty parameter and non-box inequality constraints, making the problems relatively more difficult to solve. Moreover, even if the feasibility of~\eqref{eq:nlp} is given, it may take significantly longer to solve multiple instances of~\eqref{eq:ALM} as compared to a single iteration of QPM with a large penalty, which is more suitable for MPC problems where slight suboptimality may be preferable to a longer computation time. However, in the presence of \textit{hard} general (nonlinear) inequality constraints where QPM might not be applicable, using BLM for feasible problems with the proposed solver and sparsity exploiting methods described in the following sections may be an efficient alternative. BLM is not discussed further as the scope of this paper is limited to MPC problems with box constraints on decision variables.

\section{Optimization algorithm}\label{sec:solvers}
\subsection{Bounded-variable nonlinear least squares}\label{sec:bvnlls}
In order to efficiently solve the MPC problem~\eqref{eq:nllsbox}, it is desirable to have a solution method that benefits from warm-starting information, is robust to problem scaling, and exploits the structure of the problem.  The bounded-variable nonlinear least-squares (BVNLLS) method we propose in Algorithm~\ref{alg:bvnlls} addresses such features.
It can be seen as either an \textit{ad~hoc} primal-feasible line-search SQP algorithm~\cite{Nocedal} or an extension of the Gauss-Newton method~\cite[Section~9.2]{Bjorck} to handle box-constraints.  The Gauss-Newton approximation of the Hessian is effective for nonlinear least-squares cost functions and it only needs first-order information of the residual. Although the global convergence property of Algorithm~\ref{alg:bvnlls} follows that of line-search methods for problems with simple bounds~\cite{BertsekasMM}, we provide below an alternative proof specific to BVNLLS for an insightful overview which also justifies the backtracking rule (Steps~\ref{step:linesearch1}-\ref{step:linesearch2} of Algorithm~\ref{alg:bvnlls}), that is analogous to the \textit{Armijo condition}~\cite{Nocedal} for the choice of the step-size $\alpha$. 
\begin{algorithm}[t!]
	\caption{Bounded-Variable Nonlinear Least Squares (BVNLLS) solver}
	\begin{algorithmic}[1]
		\INPUT Bounds $p,q\in\mathbb{R}^{n_z}$, feasible initial guess $z$, $b=r(z)$, optimality tolerance $\gamma\geq0$, $c\in(0,0.5)$, $\tau~\in~(0,1)$.
		\algrule
		\STATE $J \gets \Jac_zr$ (Linearization);\vspace{1pt}\label{step:start}
		\STATE $\mathcal{L} \gets \{j|z(j)\leq p(j) \}$; $\mathcal{U} \gets \{j|z(j) \geq q(j) \}$;\vspace{2.5pt}
		\STATE $d \gets J^{\top}b$ (Compute gradient of the cost function);\STATEx $\lambda_p(j)\gets d(j), \forall j\in\mathcal{L}$; $\lambda_q(j)\gets-d(j),\forall j \in \mathcal{U}$;\label{step:lambda}\vspace{1pt}
		\STATE \textbf{if} $\lambda_p(j)\geq -\gamma, \forall j\in \mathcal{L}$ \textbf{and}  $\lambda_q(j)\geq -\gamma, \forall j\in \mathcal{U}$ \textbf{and} $|d(j)| \leq\gamma, \forall j\notin \mathcal{L}\cup\mathcal{U}$ \textbf{then go to} Step~\ref{step:end} (Stop if converged to a first-order optimal point);\label{step:termination}\vspace{1pt}
		\STATE $\Delta z \gets \argmin_{p-z \leq\Delta \hat z\leq q-z} \Vert J\Delta \hat z + b \Vert_2^2$ (search direction);\label{step:bvls} \vspace{2pt}
		\STATE $\alpha=1$; $\theta \gets c\alpha d^\top\Delta z$; $\psi \gets b^{\top}b$; $b\gets r(z+\Delta z)$; $\Phi~\gets~b^{\top}b$;\vspace{1pt}
		\label{step:linesearch1}
		\WHILE{$\Phi>\psi+\theta$} (Backtracking line search)\\
		$\alpha \gets \tau \alpha$; $\theta \gets \alpha \theta$;\\
		$b \gets r(z+\alpha\Delta z)$; $\Phi \gets b^{\top}b$;
		\ENDWHILE \label{step:linesearch2}
		\STATE $z \gets z+\alpha\Delta z$; \textbf{go to} Step~\ref{step:start} (Update the iterate);\label{step:update}
		\STATE $z^{\star}\gets z$; $\lambda_p(j)\gets 0, \forall j\notin \mathcal{L}$; $\lambda_q(j) \gets 0, \forall j\notin \mathcal{U}$; \label{step:end}				
		\STATE \textbf{end.}
		\algrule
		\OUTPUT Local minimum $z^{\star}$ of~\eqref{eq:nlp}, objective function value $\Phi$ at $z^{\star}$, and Lagrange multiplier vectors $\lambda_p$ and $\lambda_q$  corresponding to lower and upper bounds, respectively.
	\end{algorithmic}
	\label{alg:bvnlls}
\end{algorithm}

\subsection{Global convergence}
At the $i{\text{th}}$ iteration of Algorithm~\ref{alg:bvnlls}, the search direction $\Delta z^{(i)}$ is computed at Step~\ref{step:bvls} as 
\begin{equation}\label{eq:bvls}
	\Delta z^{(i)} = \argmin_{\bar p \leq\Delta \hat z\leq \bar q} \Vert J \Delta \hat z + b \Vert_2^2,
\end{equation}
where the Jacobian matrix $J =\Jac_z r\left(z^{(i-1)}\right)$ is full rank, $b = r\left(z^{(i-1)}\right)$, $\bar p = p - z^{(i-1)}$ and $\bar q = q - z^{(i-1)}$, and $p,q$ are the bounds on $z$. 
\begin{lemma}[Primal feasibility]\label{lem:primfeas}
	Consider that $z^{(i)} = z^{(i-1)}+\alpha \Delta z^{(i)}$ as in Step~\ref{step:update} at the $i{\text{th}}$ iteration of Algorithm~\ref{alg:bvnlls} with any $\alpha \in (0,~1]$. If $p\leq z^{(0)} \leq q$ and $\bar{p} \leq \Delta z^{(i)} \leq \bar{q}$, then $p \leq z^{(i)} \leq q$ at all iterations $i$.
\end{lemma}
\begin{proof}
	We prove the lemma by induction. The lemma clearly holds for $i=0$, as
	by assumption the initial guess $z^0$ is feasible, $p\leq z^{(0)} \leq q$. 
	Consider the $i{\text{th}}$ iteration of Algorithm~\ref{alg:bvnlls}.
	From Step~\ref{step:bvls} we have that $p-z^{(i-1)}\leq \Delta z^{(i)}\leq q-z^{(i-1)}$, which multiplied by $\alpha$, $
	\alpha>0$, gives
	\begin{equation}
		\alpha p - \alpha z^{(i-1)} \leq \alpha \Delta z^{(i)} \leq \alpha q - \alpha z^{(i-1)}. 
		\label{eq:barpz1q}
	\end{equation}
	By adding $z^{(i-1)}$ to each side of the inequalities in~\eqref{eq:barpz1q} 
	we get
	\begin{align}\label{eq:pz1qalpha}
		\alpha p + (1 - \alpha) z^{(i-1)} \leq z^{(i)} \leq \alpha q + (1 - \alpha) z^{(i-1)}.
	\end{align}
	By induction, let us assume that $p \leq z^{(i-1)} \leq q$.
	Since $\alpha\leq 1$, we get the further inequalities
	\begin{alignat*}
		\alpha p + (1 - \alpha) p &\leq z^{(i)} &&\leq \alpha q + (1 - \alpha) q
	\end{alignat*}
	or $p\leq z^{(i)}\leq q$.
\end{proof}	
\begin{lemma}\label{lem:dirder}
	The search direction $\Delta z^{(i)}$ given by~\eqref{eq:bvls} is a descent direction for the cost function $f(z) = \frac{1}{2}\Vert r(z)\Vert_2^2$ in~\eqref{eq:nllsbox}.
\end{lemma}
\begin{proof}
	If $\mathcal{D}\left(f(z),~\Delta z\right)$ denotes the directional derivative of $f(z)$ in the direction $\Delta z$, then $\Delta z^{(i)}$ is a descent direction if $\mathcal{D}\left(f\left(z^{(i-1)}\right),~\Delta z^{(i)}\right) < 0$. By definition of directional derivative~\cite[Appendix A]{Nocedal},
	\begin{equation}\label{eq:dirder}
		\mathcal{D}\left(f\left(z^{(i-1)}\right),~\Delta z^{(i)}\right) = \nabla_z f\left(z^{(i-1)}\right)^\top \Delta z^{(i)}.
	\end{equation}
	By substituting
	\begin{equation}
		\nabla_z f\left(z^{(i-1)}\right) = \Jac_z r\left(z^{(i-1)}\right)^\top r\left(z^{(i-1)}\right) = J^\top b\label{eq:gradf}
	\end{equation}
	in~\eqref{eq:dirder} we get
	\begin{equation}\label{eq:dirdersimple}
		\mathcal{D}\left(f\left(z^{(i-1)}\right),~\Delta z^{(i)}\right) = b^\top J \Delta z^{(i)}.
	\end{equation}
	As $\Delta z^{(i)}$ solves the convex subproblem~\eqref{eq:bvls}, the following Karush-Kuhn-Tucker (KKT) conditions~\cite{mpcbook} hold:
	\begin{subequations}\label{eq:kkt}
		\begin{flalign}
			J^{\top}\left(J \Delta z^{(i)} + b\right) + \Lambda_{\bar{q}}  - \Lambda_{\bar{p}} &= \bm{0}\label{eq:gradL}\\
			\Delta z^{(i)} &\geq \bar{p}\label{eq:lb}\\
			\Delta z^{(i)} &\leq \bar{q}\label{eq:ub}\\
			\Lambda_{\bar{q}}, \Lambda_{\bar{p}} &\geq \bm{0}\label{eq:lambdageq0}\\
			\Lambda_{\bar{q}}(j)\left(\Delta z^{(i)}(j) - \bar{q}(j)\right) &= 0~\forall j\label{eq:compslack1}\\
			\Lambda_{\bar{p}}(j)\left( \bar{p}(j)-\Delta z^{(i)}(j)\right) &= 0~\forall j,\label{eq:compslack2}
		\end{flalign}
	\end{subequations}
	where $\Lambda_{\bar{q}}$ and $\Lambda_{\bar{p}}$ denote the optimal Lagrange multipliers of subproblem~\eqref{eq:bvls}. From~\eqref{eq:gradL} we have,
	\begin{equation}\label{eq:dirderDz}
		b^\top J \Delta z^{(i)} = \left(\Lambda_{\bar p} - \Lambda_{\bar q}\right)^{\top}\Delta z^{(i)}-\Delta {z^{(i)}}^\top J^\top J \Delta z^{(i)}.
	\end{equation}
	By substituting $\bar p = p - z^{(i-1)}$ and $\bar q = q - z^{(i-1)}$  in the complementarity conditions~\eqref{eq:compslack1}-\eqref{eq:compslack2}, we can write
	\begin{align*}
		&\Lambda_{\bar q}^\top \left(\Delta z^{(i)} - q +z^{(i-1)}\right) + \Lambda_{\bar p}^\top \left(p-z^{(i-1)}-\Delta z^{(i)}\right) = 0, \\
		&\text{i.e., } (\Lambda_{\bar q}-\Lambda_{\bar p} )^{\top}\Delta z^{(i)} =  \Lambda_{\bar{q}}^\top (q-z^{(i-1)}) + \Lambda_{\bar{p}}^\top (z^{(i-1)}-p).
	\end{align*}
	From~\eqref{eq:lb}-\eqref{eq:lambdageq0} we have $\Lambda_{\bar q},\Lambda_{\bar p} \geq\bm{0}$, and by Lemma~\ref{lem:primfeas} $q-z^{(i-1)}\geq \bm{0}$ as well as $z^{(i-1)}-p\geq \bm{0}$, which implies that
	\begin{align}
		(\Lambda_{\bar q}-\Lambda_{\bar p} )^{\top}\Delta z^{(i)} &\geq \bm{0}, \text{i.e.,}\notag\\
		(\Lambda_{\bar p}-\Lambda_{\bar q} )^{\top}\Delta z^{(i)} & \leq \bm{0}\label{eq:nonnegltdz}.
	\end{align}
	Since $J$ is full rank, $J^\top J>0$. Using this fact and Lemma~\ref{lem:dzne0} along with~\eqref{eq:nonnegltdz} in~\eqref{eq:dirderDz} gives
	\begin{align}\label{eq:dirderneg}
		b^\top J \Delta z^{(i)} < 0.
	\end{align}
	Considering~\eqref{eq:dirderneg}~and~\eqref{eq:dirdersimple}, we have that the directional derivative for the considered search direction is negative, which proves the lemma.
\end{proof}
\begin{remark}\label{rem:primfeas} We infer from Lemma~\ref{lem:primfeas} and~\eqref{eq:lb}-\eqref{eq:ub} that BVNLLS is a primal-feasible method, which is an important property when the function $r(z)$ is not analytic beyond bounds~\cite{bellavia}. \end{remark} 

\begin{lemma}\label{lem:zk_1kkt}
	If the solution of~\eqref{eq:bvls} is $\Delta z^{(i)} = \bm{0}$, $z^{(i-1)}$ is a stationary point satisfying the first-order optimality conditions of problem~\eqref{eq:nllsbox}. 
\end{lemma}
\begin{proof}
	Given $\Delta z^{(i)} = \bm{0}$, we need to prove that $z^{(i-1)}$ satisfies the following first-order optimality conditions for  problem~\eqref{eq:nllsbox}: 
	\begin{subequations}\label{eq:kktnlp}
		\begin{align}
			\Jac_z r(z)^\top r(z) + \lambda_q -\lambda_p &= \bm{0}\label{eq:LgradNLP}\\
			p\leq z &\leq q \label{eq:primfeas}\\
			\lambda_q, \lambda_p &\geq \bm{0}\label{eq:lambdaNLPgeq0}\\
			\lambda_q(j)(z(j) - q(j)) = \lambda_p(j)( p(j)- z(j)) &= 0,\ \forall j,
		\end{align}
	\end{subequations}
	where the optimal Lagrange multipliers are denoted by $\lambda_p$ and $\lambda_q$ for the lower and upper bounds, respectively. 
	
	By substituting $\Delta z^{(i)} = \bm{0}$ in~\eqref{eq:kkt}, and recalling $\bar{q} = q-z^{(i-1)}$ and $\bar{p}=p-z^{(i-1)}$, we obtain
	\begin{subequations}
		\label{eq:kktdz0}
		\begin{flalign}
			J^{\top}b + \Lambda_{\bar{q}} - \Lambda_{\bar{p}} &= \bm{0},\\
			p \leq z^{(i-1)} &\leq q,\\
			\Lambda_{\bar q}(j)(z^{(i-1)}(j) - q(j)) &= 0~\forall j,\\
			\Lambda_{\bar p}(j)( p(j)- z^{(i-1)}(j)) &= 0~\forall j.
		\end{flalign}
	\end{subequations}
	Clearly, considering~\eqref{eq:lambdageq0} along with the definitions of $J$, $b$, and~\eqref{eq:kktdz0}, we conclude that $z^{(i-1)}$, $\Lambda_{\bar{q}}$ and $\Lambda_{\bar{p}}$ solve the KKT system~\eqref{eq:kktnlp}. 
\end{proof}	
\begin{lemma}\label{lem:dzne0}
	In Algorithm~\ref{alg:bvnlls}, $\Delta z^{(i)} \ne \bm{0}$ at any iteration.

\end{lemma}
\begin{proof}
	We prove this lemma by contradiction. Assume that Algorithm~\ref{alg:bvnlls} reaches an iteration $i$ where Step~\ref{step:bvls} is executed and returns $\Delta z^{(i)} = \bm{0}$. This implies that $z^{(i-1)}$ is a stationary point satisfying the first-order optimality conditions of nonlinear problem~\eqref{eq:nllsbox}, as shown in Lemma~\ref{lem:zk_1kkt}. Then, the termination criterion in Step~\ref{step:termination} would end the algorithm without further computations, so that iteration $i$ is never reached, a contradiction. Note that in particular, if the initial guess $z^{(0)}$ is optimal, $\Delta z^{(i)}$ is never computed.
\end{proof}

\begin{theorem}[Global convergence of BVNLLS]\label{thm:converg}
	Consider the optimization problem~\eqref{eq:nllsbox} and define the scalar cost function $f(z) = \frac{1}{2}\Vert r(z) \Vert_2^2$. At each iteration $i$ of Algorithm~\ref{alg:bvnlls}, there exists a scalar $\alpha \in (0,1]$
	such that 
	\begin{equation}
		f\left(z^{(i-1)}+\alpha \Delta z^{(i)}\right)-f\left(z^{(i-1)}\right)< c\alpha \nabla f\left(z^{(i-1)}\right)^\top \Delta z^{(i)} 
		\label{eq:descent}
	\end{equation} 
	with $0<\alpha\leq1$ and $0<c<1$,
	where $z^{(i)}=z^{(i-1)} + \alpha\Delta z^{(i)}$.
\end{theorem}
\begin{proof}

	Consider the Taylor series expansion of $f\left(z^{(i)}\right)$
	\begin{multline}
		f\left(z^{(i-1)}+\alpha \Delta z^{(i)}\right) = f\left(z^{(i-1)}\right) + \alpha\nabla_z f\left(z^{(i-1)}\right)^{\top}\Delta z^{(i)} \\+ \frac{\alpha^2}{2} \Delta {z^{(i)}}^{\top}\nabla_z^2 f\left(z^{(i-1)}\right) \Delta z^{(i)} + \mathcal{E}(\Vert \alpha\Delta z^{(i)}\Vert^3),\label{eq:taylor}
	\end{multline}
	where the term $\mathcal{E}{\Vert(\cdot)\Vert^3}$ represents the third order error. Also,
	\begin{multline}
		\nabla^2_z f\left(z^{(i-1)}\right) = \sum_{j=1}^{n_r}r_j\left(z^{(i-1)}\right)\nabla_z^2 r_j\left(z^{(i-1)}\right) \\+ \Jac_z r\left(z^{(i-1)}\right)^\top \Jac_z r\left(z^{(i-1)}\right) = H + J^\top J, \label{eq:hessf}
	\end{multline}
	where $r_j$ denotes the $j{\text{th}}$ element of the residual vector. 
	By substituting the relations~\eqref{eq:gradf} and~\eqref{eq:hessf} in~\eqref{eq:taylor} we get
	\begin{multline}
		f\left(z^{(i-1)}+\alpha \Delta z^{(i)}\right) - f\left(z^{(i-1)}\right) = \alpha b^\top J \Delta z^{(i)} \\+ \frac{\alpha^2}{2} \Delta {z^{(i)}}^{\top} \left(H+ J^\top J\right) \Delta z^{(i)} + \mathcal{E}\left(\left\Vert \alpha\Delta z^{(i)}\right\Vert^3\right)\label{eq:taylorJ}.
	\end{multline}
	Using~\eqref{eq:dirderDz}, Equation~\eqref{eq:taylorJ} can be simplified as
	\begin{multline}\label{eq:taylorJsimplified}
		f\left(z^{(i-1)}+\alpha \Delta z^{(i)}\right) - f\left(z^{(i-1)}\right)=\\-\frac{\alpha(2-\alpha)}{2}\Delta {z^{(i)}}^\top J^\top J \Delta z^{(i)} + \alpha(\Lambda_{\bar{p}} - \Lambda_{\bar{q}})^\top \Delta z^{(i)}\\ + \frac{\alpha^2}{2}\Delta {z^{(i)}}^{\top} H  \Delta z^{(i)} +\mathcal{E}\left(\left\Vert \alpha\Delta z^{(i)}\right\Vert^3\right).
	\end{multline}
	Referring~\eqref{eq:gradf} and~\eqref{eq:dirderDz}, on subtracting $c\alpha \nabla f(z^{(i-1)})^{\top}\Delta z$ from both sides of~\eqref{eq:taylorJsimplified} we get
	\begin{multline}\label{eq:taylorJlambda}
		f\left(z^{(i-1)}+\alpha \Delta z^{(i)}\right)-f\left(z^{(i-1)}\right) - c\alpha \nabla f\left(z^{(i-1)}\right)^\top \Delta z^{(i)}  \\=-\frac{\alpha(2-2c-\alpha)}{2}\Delta {z^{(i)}}^\top J^\top J \Delta z^{(i)} \\+ \alpha(1-c)(\Lambda_{\bar{p}} - \Lambda_{\bar{q}})^\top \Delta z^{(i)}\\ + \frac{\alpha^2}{2}\Delta {z^{(i)}}^{\top} H  \Delta z^{(i)} +\mathcal{E}\left(\left\Vert \alpha\Delta z^{(i)}\right\Vert^3\right).
	\end{multline}
	Let
	\[
	\bar N = -\mbox{$\frac{(2-2c-\alpha)}{2}$}\Delta {z^{(i)}}^\top J^\top J \Delta z^{(i)}+(1-c)(\Lambda_{\bar{p}} - \Lambda_{\bar{q}})^\top \Delta z^{(i)}.
	\]
	From~\eqref{eq:nonnegltdz}, Lemma~\ref{lem:dzne0}, and from the facts that $\alpha\in(0,~1]$, $c\in(0,1)$, and that matrix $J$ has full rank ($J^\top J > 0$), we infer that $\bar N$ must be negative for sufficiently small $\alpha$. Let 
	\[ 
	\bar M = \frac{1}{2}\Delta {z^{(i)}}^{\top} H  \Delta z^{(i)} +\mathcal{E}\left(\left\Vert \alpha\Delta z^{(i)}\right\Vert^3\right)
	\]
	Then~\eqref{eq:taylorJlambda} can be written as
	\begin{multline}
		f\left(z^{(i-1)}+\alpha \Delta z^{(i)}\right) - f\left(z^{(i-1)}\right) - c\alpha \nabla f\left(z^{(i-1)}\right)^\top \Delta z^{(i)} \\=\alpha \bar N +\alpha^2 \bar M. \label{eq:reducedtaylor}
	\end{multline} 
	Let $\alpha \bar N + \alpha^2 \bar M + \epsilon = 0$, or $\epsilon = \alpha \left(-\alpha \bar M - \bar N\right)$.
	Clearly, since $\bar N < 0$, there exists a value of $\alpha>0$ such that $\epsilon>0$.  This proves that there exists a positive value of $\alpha$ such that $\alpha \bar N + \alpha^2 \bar M  < 0$. Hence from~\eqref{eq:reducedtaylor},\\
	$
		f\left(z^{(i-1)}+\alpha \Delta z^{(i)}\right) - f\left(z^{(i-1)}\right) - c\alpha \nabla f\left(z^{(i-1)}\right)^\top \Delta z^{(i)}  <0$,	for a sufficiently small positive value of $\alpha$.
\end{proof}	

\begin{figure*}[t!]
	\addtocounter{equation}{1}
	\begin{equation}
		\label{eq:spyJ}
		\begin{array}{l} 
			\Jac{h_k}(z) = \nabla_z h_k(z_k,\phi_k)^\top =\\[8pt]

			\resizebox{0.95\hsize}{0.126\vsize}{$\left[\begin{array}{ccccccccc|ccccc}
					B_1^{(1)} & A_0^{(1)} & \bm{0} & \bm{0} &\cdots&&&  & &&  \cdots  & && \bm{0}\\
					B_2^{(2)} & A_1^{(2)} & B_1^{(2)} & A_0^{(2)} & \bm{0} &\cdots & &&   &&&    \cdots  && \bm{0}\\[8pt]
					&\vdots && \ddots &&  && \ddots &&&&&& \vdots \\[8pt]
					B_{N_{\mathrm{u}}}^{(N_{\mathrm{u}})} & A_{N_{\mathrm{u}}-1}^{(N_{\mathrm{u}})} &  & &  & \cdots&& B_1^{(N_{\mathrm{u}})} & A_0^{(N_{\mathrm{u}})} & \bm{0} & \cdots & && \bm{0}\\[4pt] \hline &&&&&&&&&&&&&\\[-2pt]
					B_{N_{\mathrm{u}} +1}^{(N_{\mathrm{u}}+1)} & A_{N_{\mathrm{u}}}^{(N_{\mathrm{u}}+1)} &   & \cdots  & & B_3^{(N_{\mathrm{u}}+1)} & A_2^{(N_{\mathrm{u}}+1)} & \sum\limits_{i=1}^{2}B_i^{(N_{\mathrm{u}}+1)} & A_1^{(N_{\mathrm{u}}+1)} & A_0^{(N_{\mathrm{u}}+1)} & \bm{0} & \cdots& &\bm{0}\\
					B_{N_{\mathrm{u}}+2}^{(N_{\mathrm{u}}+2)} & A_{N_{\mathrm{u}}+1}^{(N_{\mathrm{u}}+2)} & \ddots& & \cdots & B_4^{(N_{\mathrm{u}}+2)} & A_3^{(N_{\mathrm{u}}+2)} &  \sum\limits_{i=1}^{3}B_i^{(N_{\mathrm{u}}+2)} & A_2^{(N_{\mathrm{u}}+2)} & A_1^{(N_{\mathrm{u}}+2)} & A_0^{(N_{\mathrm{u}}+2)} & \bm{0} & \cdots & \bm{0}\\[8pt]
					&\vdots && \ddots  &&\ddots & &\vdots& &\vdots&& \ddots & &\bm{0}\\[8pt]
					B_{N_{\mathrm{p}}}^{(N_{\mathrm{p}})} & A_{N_{\mathrm{p}}-1}^{(N_{\mathrm{p}})} & & \cdots &  & B_{N_{\mathrm{p}}-N_{\mathrm{u}}+2}^{(N_{\mathrm{p}})}& A_{N_{\mathrm{p}}-N_{\mathrm{u}}+1}^{(N_{\mathrm{p}})} &  \sum\limits_{i=1}^{N_{\mathrm{p}}-N_{\mathrm{u}}+1}B_i^{(N_{\mathrm{p}})} & A_{N_{\mathrm{p}}-N_{\mathrm{u}}}^{(N_{\mathrm{p}})} & A_{N_{\mathrm{p}}-N_{\mathrm{u}}-1}^{(N_{\mathrm{p}})} & \cdots &  &  A_1^{(N_{\mathrm{p}})} & A_0^{(N_{\mathrm{p}})}
				\end{array}\right]$}
		\end{array}
	\end{equation}
\end{figure*}
\addtocounter{equation}{-2}
\begin{remark}
	In the case of linear MPC i.e., when $h_k(z_k,~\phi_k)$ is linear in~\eqref{eq:nllsbox}, the bounded-variable least-squares (BVLS) problem~\eqref{eq:bvls} is solved only once as the KKT conditions~\eqref{eq:kktnlp} coincide with~\eqref{eq:kkt}. Moreover, the backtracking steps are not required as the higher order terms in~\eqref{eq:taylor} are zero and Theorem~\ref{thm:converg} holds with $\alpha=1$ for any $c\in(0,~1)$.
\end{remark} 
\begin{remark} Referring to~\eqref{eq:taylorJlambda}, the value of $c$ is practically kept below $0.5$ in Algorithm~\ref{alg:bvnlls} in order to enforce fast convergence with full steps and is typically chosen to be as small as $10^{-4}$~\cite{Nocedal}. As seen in~\eqref{eq:taylorJlambda}, since we only need the matrix $J$ to be full rank for convergence of BVNLLS, the matrix $W_k$ in~\eqref{eq:nllsbox} may be rank-deficient as long as $J$ is full rank.
\end{remark} \begin{remark} Suboptimality in solving the BVLS subproblems may result in a smaller decrease in the cost between BVNLLS iterations than the theoretical decrease indicated by Theorem~\ref{thm:converg}. Hence, it is essential to have an accurate BVLS solver in order to have fast convergence.
	For this reason, we advocate the use of active-set methods to solve
	BVLS problems.\end{remark}

Each iteration of BVNLLS corresponds to solving a linear MPC problem, a special case of~\eqref{eq:nllsbox}. This allows to have a common framework for linear and nonlinear MPC in our approach. The BVLS problem~\eqref{eq:bvls} can be solved efficiently and accurately by using a primal active-set algorithm as shown in~\cite{bvlsTN}, which uses numerically robust recursive QR factorization routines to solve the LS subproblems. Unlike most of the QP solvers, in which the Hessian $J^\top J$ would be constructed via a matrix multiplication and then factorized, the BVLS solver~\cite{bvlsTN} only factorizes column subsets of $J$, whose condition number is square-root as compared to that of $J^\top J$, which makes it numerically preferable. In applications with very restrictive memory requirements, using the methods described in Section~\ref{sec:abstractmat} with the gradient-projection algorithm~\cite{nesterov} on the primal problem~\eqref{eq:bvls}, one may employ a matrix-free solver similar to~\cite{matrixfreeopt} and its references. However, when using the gradient-projection algorithm, its low memory usage may come at the cost of slow convergence due to their sensitivity to problem scaling. The following sections show how the Jacobian matrix $J$ can be replaced by using parameterized operators for saving memory and how its sparsity can be exploited for faster execution of the proposed BVLS solver of~\cite{bvlsTN}.
\section{Abstracting matrix instances}\label{sec:abstractmat}
\subsection{Problem structure}\label{sec:probstruct}
The sparse structure of matrices $W_k$ and $\nabla_z h_k(z_k,~\phi_k)^\top$, which form the Jacobian $J$ of the residual in~\eqref{eq:nllsbox}, completely depends on  the MPC tuning parameters, model order, and the ordering of the decision variables. 

Let us assume that there are no slack variables due to non-box inequality constraints~\eqref{eq:genineq}. In case of slack variables, the sparsity pattern will depend on the structure of Jacobian of the inequality constraints, which is not discussed in further detail here for conciseness. 
By ordering the decision variables in vector $z_k$ as follows
\begin{equation}
	\label{eq:zvec}
	\begin{array}{rcl}
		z_k &=& \left[u_k^{\top}\ y_{k+1}^{\top}\ u_{k+1}^{\top}\ y_{k+2}^{\top}\ \ldots\ u_{k+N_{\mathrm{u}}-1}^{\top}\ y_{k+N_{\mathrm{u}}}^{\top}\ \vline\right.\\
		&&\quad \left.  y_{k+N_{\mathrm{u}}+1}^{\top}\ \ldots\ y_{k+N_{\mathrm{p}}-1}^{\top}\ y_{k+N_{\mathrm{p}}}^{\top}\right]^{\top}
	\end{array} 
\end{equation}
we get the matrix structure described in~\eqref{eq:spyJ},
where the superscript of matrices in parentheses denote the output prediction
step the matrices refer to. Note that we dropped the parentheses $(S_k)$ in~\eqref{eq:spyJ} to simplify the notation and, as defined in~\eqref{eq:arxmimovec}, $A(S_k)_j = \bm{0}$, $\forall j>n_{\mathrm{a}}$, and $B(S_k)_j = \bm{0}$, $\forall j > n_{\mathrm{b}}$. Clearly, the Jacobian matrix
$\Jac{h_k}$ of equality constraints only consists of entries from the sequence of linear models of the form~\eqref{eq:arxmimovec} linearized around the initial guess trajectory. 
\addtocounter{equation}{1}
\begin{figure}[t!]
	\centering
	\includegraphics[width=\linewidth]{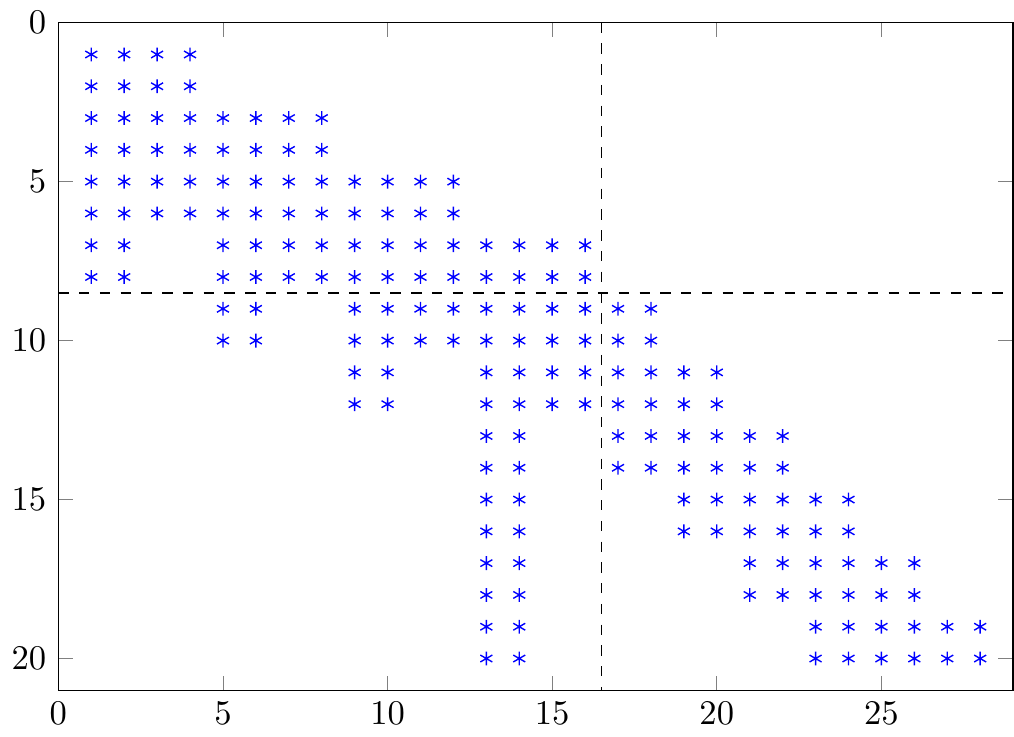}\vspace{-5pt}
	\caption{Sparsity pattern of $\Jac{h_k}$ for a random model with $N_{\mathrm{p}}=10$, $N_{\mathrm{u}}=4$, $n_{\mathrm{a}}=2$, $n_{\mathrm{b}}=4$, $n_{u}=2$ and $n_{y}=2$.}
	\label{fig:spyJhk}
\end{figure}
Considering the model parameters $n_{\mathrm{a}}, n_{\mathrm{b}}$ to be smaller than 
$N_{\mathrm{p}}$ in~\eqref{eq:spyJ}, as illustrated in Figure~\ref{fig:spyJhk}, we observe that the top-left part of $\Jac{h_k}$ is block sparse, the bottom-right part has a block-banded structure, the bottom-left part has dense columns corresponding to $u_{k+N_{\mathrm{u}}-1}$, whereas the top-right part is a zero matrix with  $n_y N_{\mathrm{u}}$ rows and $n_y\cdot(N_{\mathrm{p}}-N_{\mathrm{u}})$ columns. If $n_{\mathrm{a}}, n_{\mathrm{b}}$ are greater than $N_{\mathrm{p}}$, then $\Jac{h_k}$ would instead have its bottom-left part to be dense with block lower-triangular structure in its top-left and bottom-right parts. All in all, the sparsity pattern of $\Jac{h_k}$ is completely defined by the model parameters $n_u$, $n_y$, $n_{\mathrm{a}}$, $n_{\mathrm{b}}$, and MPC horizons $N_{\mathrm{u}}$, $N_{\mathrm{p}}$.
Clearly, evaluating $\Jac{h_k}$ only requires the sequence of linear models 
and the sparsity pattern information. Note that in case the linear models are computed
by a linearization function, a memory/throughput tradeoff
can be chosen here, as they can be either computed once and stored (lowest throughput), or evaluated by the linearization each time they are required
(lowest memory allocation).
Finally, recalling~\eqref{eq:nllsbox}, we obtain the full Jacobian matrix 
\[
J=\left[\begin{matrix}W_k\\\Jac{h_k} \end{matrix}\right]
\]
required in Algorithm~\ref{alg:bvnlls}, where $W_k$ is the
block diagonal matrix
\begin{eqnarray*}
	W_k&=&\blkdiag (W_{u_{k}},W_{y_{k+1}},W_{u_{k+1}},W_{y_{k+2}},\ldots,\\
	&&W_{u_{k+N_{\mathrm{u}}-1}},W_{y_{k+N_{\mathrm{u}}}},W_{y_{k+N_{\mathrm{u}}+1}},\ldots,W_{y_{k+N_{\mathrm{p}}}})
\end{eqnarray*}
In the sequel we assume for simplicity that all matrices $W_{u_{(\cdot)}},W_{y_{(\cdot)}}$ are diagonal, so that $W_k$ is actually a diagonal matrix.

\subsection{Abstract operators}\label{sec:absops}
All matrix-vector operations involving $J$ in Algorithm~\ref{alg:bvnlls}
and in the BVLS solver~\cite{bvlsTN}, including the matrix factorization routines that will be described in  Section~\ref{sec:sparseQR}, only need the product of a column subset of $J$ or a row-subset of $J^\top$ with a vector. 
Hence, rather than explicitly forming and storing $J$, all the operations involving $J$ can be represented by two operators \texttt{Jix} ($i$th column of $J$ times a scalar $x$) and \texttt{JtiX} ($i$th column of $J$ times a vector $X$) defined by Algorithms~\ref{alg:Jxi} and~\ref{alg:Jtxi}, respectively.

\begin{algorithm}[t!]
	\caption{Operator \texttt{Jix}}
	\begin{algorithmic}[1]
		\INPUT Output memory $v=\bm{0}\in \mathbb{R}^{n_z + N_{\mathrm{p}}n_y}$; 
		
		vector $w$ storing diagonal elements of $W_k$; scalar $x$; column number $i$; parameters $n_a,~n_b,~n_u,~n_y,~N_{\mathrm{u}}$ and $N_{\mathrm{p}}$.
		\algrule
		\STATE $v(i) \gets w(i)\cdot x$;	\label{step:firstnz}
		\STATE Find integers $\beta \in [0, N_{\mathrm{p}})$ and $\eta \in [1,~n_u+n_y]$ such that $i=\beta n_y + n_u\cdot\min(\beta,~N_{\mathrm{u}}-1) + \eta$;\label{step:findbeta}
		\STATE $\bar n \gets N_{\mathrm{u}}n_u+(N_{\mathrm{p}}+\beta)n_y$; $m \gets N_{\mathrm{u}}n_u + 2N_{\mathrm{p}}n_y$; $j \gets 0$;\label{step:barnm}
		
		\IF{$\beta \ne N_{\mathrm{u}}-1$ \textbf{or} $\eta>n_u$}
		\STATE \textbf{if } $\eta>n_u$, $\bar m \gets \bar n + n_a n_y +n_y$ \textbf{ else} $\bar m \gets \bar n + n_b n_y$; \label{step:barm}
		\FOR {$j'\in\{\bar n,\bar n+n_y,\cdots,\min(\bar m, m)-n_y\}$}\label{step:for1}
		\IF{$\eta > n_u$} $\forall j''\in\{1,\cdots,n_y\}$,
		\STATE		 $v(j'+j'') \gets x\cdot A^{(\beta + j+1)}_j(j'',~\eta - n_u)$; \label{step:etaoutput}
		\ELSE \STATE $v(j'+j'') \gets x\cdot B^{(\beta + j+1)}_{j+1}(j'',~\eta)$; \label{step:etainput}
		\ENDIF
		\STATE $j\gets j+1$;
		\ENDFOR
		\ELSE
		\FOR{$j'\in\{\bar n,\bar n+n_y,\cdots,m-n_y\}$}\label{step:for2}
		\STATE $j \gets j+1$;\label{step:formBbar1}
		\STATE $\bar B(j'') \gets \sum\limits_{i'=1}^{\min(j,~n_b)} B^{\beta+j}_{i'}(j'',~\eta)$, $\forall j''\in[1,n_y]$;\label{step:formBbar2}
		\STATE $v(j'+j'') \gets x\cdot \bar B(j'')$, $\forall j''\in[1, n_y]$;
		\ENDFOR
		\ENDIF
		\STATE \textbf{end.}
		\algrule
		\item[\textbf{Output:}] Vector $v=i{\text{th}}$ column of $J$ in~\eqref{eq:bvls} scaled by $x$.
	\end{algorithmic}
	\label{alg:Jxi}
\end{algorithm}
\begin{algorithm}[t!]
	\caption{Operator \texttt{JtiX}}
	\begin{algorithmic}[1]
		\INPUT Vector $w$ storing diagonal elements of $W_k$; vector $X$; column number $i$; parameters $n_a,~n_b,~n_u,~n_y,~N_{\mathrm{u}}$ and $N_{\mathrm{p}}$.
		\algrule
		\STATE $v' \gets w(i)\cdot X(i)$;	
		\STATE Steps~\ref{step:findbeta}-\ref{step:barnm} of Algorithm~\ref{alg:Jxi};
		\IF{$\beta \ne N_{\mathrm{u}}-1$ \textbf{or} $\eta>n_u$}
		\STATE \textbf{if } $\eta>n_u$, $\bar m \gets \bar n + n_a n_y + n_y$ \textbf{ else } $\bar m \gets \bar n + n_b n_y$;
		\FOR {$j'\in\{\bar n,\bar n+n_y,\cdots,\min(\bar m, m)-n_y\}$}
		\IF{$\eta > n_u$} $\forall j''\in\{1,\cdots,n_y\}$,
		\STATE		 $v' \gets v'+ X(j'+j'')\cdot A^{(\beta + j+1)}_j(j'',~\eta - n_u)$; \label{step:diffstep1}
		\ELSE \STATE $v' \gets v'+X(j'+j'')\cdot B^{(\beta + j+1)}_{j+1}(j'',~\eta)$; \label{step:diffstep2}
		\ENDIF
		\STATE $j\gets j+1$;
		\ENDFOR
		\ELSE
		\FOR{$j'\in\{\bar n,\bar n+n_y,\cdots,m-n_y\}$}
		\STATE Steps~\ref{step:formBbar1}-\ref{step:formBbar2} of Algorithm~\ref{alg:Jxi};
		\STATE $v' \gets v'+X(j'+j'')\cdot \bar B(j'')$, $\forall j''\in[1, n_y]$; \label{step:diffstep3}
		\ENDFOR
		\ENDIF
		\STATE \textbf{end.}
		\algrule
		\item[\textbf{Output:}] $v'=$ inner product of $i{\text{th}}$ row of $J^{\top}$ in~\eqref{eq:bvls} and $X$.
	\end{algorithmic}
	\label{alg:Jtxi}
\end{algorithm}
The basic principle of both Algorithms~\ref{alg:Jxi} and~\ref{alg:Jtxi} is to extract \nz entries indexed in $J$ from the corresponding model coefficients based on the given model and MPC tuning parameters. Since the top part $W_k$ of $J$ is a diagonal matrix, the first \nz entry in any column of $J$ is obtained from the vector of weights (cf.\ Step~\ref{step:firstnz} of \texttt{Jix} and~\texttt{JtiX}). The remaining steps only concern evaluating $\Jac h_k$ as in~\eqref{eq:spyJ}, in which the coefficients in each column match the corresponding element in $z_k$ as in~\eqref{eq:zvec}. Referring to the sparsity pattern of $\Jac h_k$ in~\eqref{eq:spyJ}, each of its columns only contains either model coefficients related to the input or to the output, and in the columns corresponding to the inputs $u_{k+N_{\mathrm{u}}-1}$ some of the input coefficients are summed due to the finite control horizon $N_{\mathrm{u}}<N_{\mathrm{p}}$. The location of the first \nz term in each column of $\Jac h_k$ depends on the corresponding stage of the input or output variable in prediction, whereas the last entry depends on $n_a$ or $n_b$ and $N_{\mathrm{p}}$. Hence, in Step~\ref{step:findbeta} of Algorithm~\ref{alg:Jxi}, the integer $\beta$ is computed such that $\beta n_y+1$ is the index of the first \nz entry in $\Jac h_k(z)$ (cf.\ Steps~\ref{step:barnm},~\ref{step:for1}~and~\ref{step:for2}). The integer $\eta$ computed in the same step denotes the input or output channel to which the column corresponds, in order to accordingly index and extract the coefficients to be scaled as shown in Steps~\ref{step:etaoutput},~\ref{step:etainput}~and~\ref{step:formBbar2} of Algorithm~\ref{alg:Jxi}. Depending on the column index $i$ of $J$, computing $\beta$ and $\eta$ only needs a trivial number of integer operations including at most one integer division, for instance, if $i\leq N_{\mathrm{u}}(n_u+n_y)$, $\beta$ is obtained by an integer division of $i$ by $(n_u+n_y)$ and $\eta = i - \beta (n_u+n_y)$. The same computation is straightforward for the only other possible case in which $i>N_{\mathrm{u}}(n_u+n_y)$. 

Clearly, since the rows of $J^{\top}$ are the columns of $J$, Algorithm~\ref{alg:Jtxi} differs from Algorithm~\ref{alg:Jxi} only in Steps~\ref{step:diffstep1},~\ref{step:diffstep2}~and~\ref{step:diffstep3} in which the scaled coefficient is accumulated to the resulting inner product instead of a plain assignment operation. It is possible to easily extend Algorithm~\ref{alg:Jtxi} for the special case in which $X$ in \texttt{JtiX} is the $i{\text{th}}$ column of $J$ i.e., to efficiently compute the 2-norm of the $i{\text{th}}$ column of $J$, which may be required in the \la routines. Replacing the instances of $J$ by \texttt{Jix} and \texttt{JtiX} in the BVNLLS and in the inner BVLS solver has the following advantages:

\noindent 1)	The problem construction step in MPC is eliminated, as matrix $J$ is neither formed nor stored.

\noindent 2) The code of the two operators does not change with any change in the required data or dimensions as all the indexing steps are parameterized in terms of MPC tuning parameters, i.e., known data. Hence, the resulting optimization solver does not need to be code-generated with a change in problem dimensions or data. The same fact also allows real-time changes in the MPC problem data and tuning parameters without any change in the solver.  A structural change in the BVNLLS optimization problem formulation, such as the type of performance index, is already decided in the MPC design phase and can be simply accommodated by only modifying Algorithms~\ref{alg:Jxi} and~\ref{alg:Jtxi}.

\noindent 3) Unlike basic sparse-matrix storage schemes~\cite{SparseMats} which would store the \nzs of $J$ along with indexing information, we only store the sequence of linear models at most, resulting in a significantly lower memory requirement. Alternatively, as mentioned earlier, even the coefficients $A^{(*)}_{*}$, $B^{(*)}_{*}$ can be generated during the execution of Algorithms~\ref{alg:Jxi}- \ref{alg:Jtxi} using linearization functions 
applied on the current trajectory.

\noindent 4) The number of floating-point operations (\textit{flops}) involving instances of $J$, both in the BVNLLS and the BVLS solvers, is minimal and is reduced essentially to what sparse \la routines can achieve.

\noindent 5) A matrix-free implementation can be achieved when using the gradient-projection algorithm~\cite{nesterov} to solve~\eqref{eq:bvls} in BVNLLS, as the operators
\texttt{Jix} and \texttt{JtiX} can be used for computing the gradient. In addition, considering that even the model coefficients are optional to store, the resulting NMPC algorithm will have negligible increase in memory requirement w.r.t.\ the prediction horizon.

\section{Recursive thin QR factorization}\label{sec:sparseQR}
The primal active-set method for solving BVLS problems described in~\cite{bvlsTN} efficiently solves a sequence of related LS problems using recursive thin QR factorization. The reader is referred to~\cite{bvlsTN,updateGS,GolubVLoan} for an overview on thin QR factorization and the recursive update routines in the context of the BVLS solver. This section shows how the sparsity of matrix $J$ can be exploited for significantly reducing the computations involved in the recursive updates of its QR factors, without the use of sparse-matrix storage or conventional sparse \la routines. The main idea is to have the location of \nzs in the matrix factors expressed in terms of model and MPC tuning parameters,
as described above. We first analyze how the sparse structure of column subsets of $J$ is reflected in their thin QR factors based on Gram-Schmidt orthogonalization,
then characterize the recursive update routines. 

\subsection{Gram-Schmidt orthogonalization}\label{sec:ortho}
Recall that $J\in\mathbb{R}^{m\times n}$, where $n=N_{\mathrm{u}}n_u+N_{\mathrm{p}}n_y$ and $m = n+N_{\mathrm{p}}n_y$, i.e., $m > n$ (see ~\eqref{eq:nllsbox},~\eqref{eq:bvls},~\eqref{eq:zvec} and~\eqref{eq:spyJ}).  Let $J_{\mathcal{F}}$ denote the matrix formed from those columns of $J$ with indices in the set $\mathcal{F}$.  Then there exists a unique thin QR factorization~\cite[Theorem 5.2.3]{GolubVLoan} of $J_{\mathcal{F}}$ which may be expressed via the Gram-Schmidt orthonormalization procedure $\forall i\in[1, \vert \mathcal{F} \vert]$ as 
\begin{subequations}\label{eq:thinQR}
	\begin{align}
		Q'_i &= J_{\mathcal{F}_i} - \sum\limits_{j=1}^{i-1}Q_jQ_j^{\top} J_{\mathcal{F}_i},\label{eq:CGS}\\
		Q_i &= Q'_{i} / \left\Vert Q_{i}^{'}\right\Vert_2, \label{eq:qnorm}\\
		R(j,i) &= Q_j^{\top}J_{\mathcal{F}_i}, \forall j\in[1,i-1],\label{eq:Rfact}\\
		R(i,i) &= \left\Vert Q^{'}_{i}\right\Vert_2,\label{eq:Rfactdiag}
	\end{align}
\end{subequations}
where $Q\in\mathbb{R}^{m\times \vert \mathcal{F}\vert} \coloneqq [Q_1, Q_2,\cdots, Q_{\vert\mathcal{F}\vert}]$ has orthonormal columns, $R\in\mathbb{R}^{\vert\mathcal{F}\vert\times\vert \mathcal{F}\vert}$ is upper triangular and $J_{\mathcal{F}} = QR$. In~\eqref{eq:thinQR}, with a slight abuse of notation, the subscripts denote column number, i.e., 
$Q_i$ denotes the $i{\text{th}}$ column of $Q$, whereas $\mathcal{F}_i$ denotes the $i{\text{th}}$ index in $\mathcal{F}$. As shown in~\eqref{eq:CGS}, starting from the first column of $J_{\mathcal{F}}$, the procedure constructs an orthogonal basis by sequentially orthogonalizing the subsequent columns w.r.t.\ the basis. The orthogonalization procedure shown in~\eqref{eq:CGS} is referred to as the classical Gram-Schmidt (CGS) method~\cite[Section 5.2.7]{GolubVLoan}. Since the CGS method is practically prone to numerical cancellation due to finite-precision arithmetic, we use the modified Gram-Schmidt (MGS) method~\cite[Section 5.2.8]{GolubVLoan} in which the orthogonalization is performed using the working value of $Q'_i$ instead of $J_{\mathcal{F}_i}$ in each iteration of the procedure. When applying MGS to solve the linear system before recursive updates, we also orthogonalize the right hand side (RHS) of the equations, i.e., we use an augmented system of equations in order to compensate the orthogonalization error (cf.\ \cite[Chapter 19]{Trefethen}). Moreover, for numerical robustness in limited precision, in the proposed MGS procedure a reorthogonalization step is automatically performed which iteratively refines the QR factors for reducing the orthogonalization error in case it exceeds a given threshold (cf.\ \cite[Algorithm 2]{bvlsTN},~\cite{updateGS}). 

\subsection{Sparsity analysis}\label{sec:sparsityThinQR}
In order to avoid redundant \textit{flops} due to multiplying zero entries while solving the LS problems without sparse storage schemes, we first determine the sparsity pattern of $Q$ and $R$ approximately, based on the relations described in~\eqref{eq:thinQR}. While doing so, the following notions will be used.
\begin{definition}[Nonzero structure] \label{def:nzs}
	We define the \nz structure of a vector $x$ to be the set of indices $\mathcal{S}(x)$ such that $x(i) \ne 0$, $\forall i\in \mathcal{S}(x)$, and $x(j) = 0$, $\forall j \notin \mathcal{S}(x)$.
\end{definition}
\begin{definition}[Predicted \nz structure] \label{def:pnzs}
	If $\hat{\mathcal{S}} (x)$ denotes the predicted \nz structure of a vector $x$, then $x(j) = 0~ \forall j \notin \hat{\mathcal{S}}(x)$ i.e., $\hat{\mathcal{S}} (x) \supseteq \mathcal{S}(x)$.
\end{definition}
Based on the Definition~\ref{def:nzs}, $x=x'$ implies \begin{equation}
	\label{eq:SxeqSxp}
	\mathcal S(x) = \mathcal S(x').
\end{equation}
\begin{align}
	\mathcal{S}(x'+x'') \subseteq \left\{\mathcal{S}(x') \cup \mathcal{S} (x'')\right\}, \label{eq:S1plusS2}
\end{align}
which holds with equality, i.e., $\mathcal{S}(x'+x'') = \left\{\mathcal{S}(x') \cup \mathcal{S} (x'')\right\}$, if and only if the set \\$\{i\vert x'(i)+x''(i) = 0, x'(i)\ne0, x''(i) \ne 0\} = \emptyset$. Likewise,
\begin{align*}
	\mathcal S(\kappa x) \subseteq \mathcal S(x), \kappa \in \mathbb{R},
\end{align*}
because $ \mathcal S(\kappa x) = \emptyset$ for $\kappa=0$ whereas, 
\begin{align}\label{eq:Sxkappane0}
	\mathcal S(\kappa x)~=~\mathcal S(x), \forall \kappa~\in~\mathbb{R}\setminus \{0\}. 
\end{align}
\begin{theorem}\label{thm:nnzQ}
	Consider an arbitrary sparse matrix $M \in \mathbb{R}^{n_1\times n_2}$ of full rank such that $n_1\ge n_2$ and let $\tilde Q$ denote the Q-factor from its thin QR factorization i.e., $M=\tilde Q \tilde R$. The \nz structure of each column $\tilde Q_i$ of $\tilde Q$ satisfies
	\begin{subequations}
		\label{eq:nnzQ}
		\begin{align}
			\mathcal S \left(\tilde Q_i\right) &\subseteq  \bigcup\limits_{j=1}^{i} \mathcal S \left(M_j\right), \forall i \in [1, n_2],\label{eq:nnzQ1}\\
			\text{and } \mathcal{S} \left(\tilde Q _1\right) &= \mathcal S\left(M_1\right).\label{eq:nnzQ0}
		\end{align}
	\end{subequations}
\end{theorem}
\begin{proof}
	We consider the Gram-Schmidt orthogonalization procedure described in~\eqref{eq:thinQR} applied to $M$ with $\mathcal F = [1, n_2]$ (this simplifies the notation, i.e., $\mathcal{F}_i = i$). 
	Referring to~\eqref{eq:qnorm},	since $\tilde Q'$ represents an orthogonal basis of the full rank matrix $M$ with real numbers, $1/\left\Vert\tilde Q_i'\right\Vert \ne 0$ $\forall i$, and hence from~\eqref{eq:Sxkappane0},
	\begin{equation}\label{eq:sq2eqstq2}
		\mathcal S\left(\tilde Q_i\right) = \mathcal S\left(\tilde Q'_i\right), \forall i.
	\end{equation}
	From~\eqref{eq:CGS},
	\begin{equation}\label{eq:q1}
		\tilde Q'_1 = M_1.
	\end{equation}
	Thus, considering~\eqref{eq:q1} with~\eqref{eq:SxeqSxp}~and~\eqref{eq:sq2eqstq2} proves~\eqref{eq:nnzQ0}. 
	Again, considering~\eqref{eq:CGS} with~\eqref{eq:sq2eqstq2} and~\eqref{eq:SxeqSxp},
	\begin{align}\label{eq:SQip}
		\mathcal S\left(\tilde Q_i\right) = \mathcal S\left(M_i - \sum\limits_{j=1}^{i-1}\tilde Q_j \tilde Q_j^{\top} M_i\right) = \mathcal S\left(M_i + \sum\limits_{j=1}^{i-1}\tilde Q_j \kappa_j\right)
	\end{align}
	where $\kappa_j = -\tilde Q_j^{\top}M_i \in \mathbb{R}, \forall j\in[1, i-1]$, as $\kappa_j$ represents the result of an inner product of two real vectors. From~\eqref{eq:S1plusS2} and~\eqref{eq:SQip},
	\begin{align}\label{eq:SQrec}
		\mathcal S\left(\tilde Q_i\right) \subseteq \left\{\mathcal S\left(M_i\right) \cup \left\{\bigcup\limits_{j=1}^{i-1}\mathcal S\left(\tilde Q_j\right)\right\}\right\}.
	\end{align}
	Applying~\eqref{eq:SQrec} recursively,
	\begin{align}\label{eq:SQpostrec}
		\mathcal S\left(\tilde Q_i\right) \subseteq \left\{\left\{ \bigcup\limits_{j=2}^{i} \mathcal S\left(M_j\right)\right\} \cup \mathcal S\left(\tilde Q_1\right)\right\}.
	\end{align}
	Thus, substituting~\eqref{eq:nnzQ0} in~\eqref{eq:SQpostrec} completes the proof.
\end{proof}
\begin{corollary}\label{corr:nnzR}
	Given $i\in [1, n_2]$ and $j' \in [1, n_2]$,
	\begin{equation}
		\text{if } \left\{\bigcup\limits_{j=1}^{j'} \mathcal S \left(M_j\right)\right\} \cap \mathcal{S} \left(M_i\right) = \emptyset, \text{then } \tilde R(j, i) = 0~ \forall j \in [1, j']. \label{eq:nnzR}
	\end{equation}	
\end{corollary}
\begin{proof}
	Based on result~\eqref{eq:nnzQ1} of Theorem~\ref{thm:nnzQ}, we can say that $\bigcup\limits_{j=1}^{i} \mathcal S \left(M_j\right)$ is a predicted \nz structure of $\tilde Q_i$ i.e., 
	\begin{equation}
		\bigcup\limits_{j=1}^{i} \mathcal S \left(M_j\right) = \hat{\mathcal S} \left(\tilde Q_i\right), \label{eq:approxnnzQMs}
	\end{equation}
	and hence
	\begin{equation}
		\hat{\mathcal S} \left(\tilde Q_i\right) =  \mathcal S \left(M_i\right) \cup \hat{\mathcal S}\left(\tilde Q_{i-1}\right), \forall i \in [1, n_2].\label{eq:approxnnzQ}
	\end{equation}
	If $\mathcal{S}\left(\tilde Q_j\right) \cap \mathcal{S} \left( M_i \right)= \emptyset$, then $\tilde Q_j$ and $M_i$ have disjoint \nz structures and hence, referring to~\eqref{eq:Rfact}, 
	\begin{equation}\label{eq:Rji0relation1} \mathcal{S}\left(\tilde Q_j\right) \cap \mathcal{S} \left( M_i \right)= \emptyset \implies R(j,i) = \tilde Q_j ^{\top} M_i = 0.\end{equation} From~\eqref{eq:approxnnzQ} we have that \begin{equation}
		\hat{\mathcal{S}}\left(\tilde Q_i\right) \supseteq \hat{\mathcal{S}}\left(\tilde Q_{i'}\right), \forall i'<i.\label{eq:Qim1subQi}
	\end{equation} 

	From~\eqref{eq:approxnnzQMs},~\eqref{eq:Qim1subQi} and Definition~\ref{def:pnzs}, i.e.,  $\hat{\mathcal{S}}\left(\tilde Q_i\right) \supseteq \mathcal S \left(\tilde Q_i\right)$, it follows that\\ $\left\{\bigcup\limits_{j=1}^{j'} \mathcal S \left(M_j\right)\right\} \cap \mathcal{S} \left(M_i\right) = \emptyset$ implies $\hat{\mathcal{S}} \left(\tilde Q_j\right) \cap \mathcal S\left(M_i\right) = \emptyset$, $\forall j<j'$. The corollary result is then immediate given~\eqref{eq:Rji0relation1}. 
\end{proof}
Theorem~\ref{thm:nnzQ} and Corollary~\ref{corr:nnzR} establish rigorous upper bounds on the \nz structure of the QR factors based on the \nz structure of the factorized matrix. 

Since the \nz structure of $J_{\mathcal{F}}$ is completely determined in terms of model and tuning parameters as shown in Section~\ref{sec:absops}, the predicted \nz structure of its QR factors consequently depends only on them, as will be shown in the remaining part of this section.  
\begin{corollary}\label{corr:nnzJFQ}
	Consider the matrix $J\in\mathbb{R}^{m\times n}$ whose first $n$ rows form a diagonal matrix and the last $m-n$ rows contain $\Jac h_k(z)$ as shown in~\eqref{eq:spyJ}. Let $J_{\mathcal{F}}$ denote the matrix formed from the columns of $J$ indexed in the index set $\mathcal{F}$ such that $\mathcal{F}_{i+1}>\mathcal{F}_i, \forall i\in[1,\vert \mathcal{F} \vert ]$. If $Q\in\mathbb{R}^{m\times\vert\mathcal{F}\vert}$ denotes the Q-factor from the thin QR factorization of $J_{\mathcal{F}}$, then,\\ $\forall i\in[2,\vert \mathcal{F} \vert ]$, $\left\{\bigcup\limits_{j=1}^{i}\left\{\mathcal{F}_j\right\}\right\} \cup \left(\bar n_{\mathcal{F}_1},  \max\left(\mathcal{B}_{i-1}, \min\left(\bar m_{\mathcal{F}_i}, m\right)\right)\right] =\hat{\mathcal S} \left(Q_i\right)$, where the positive integers $\bar n_{j'}$, $\bar m_{j'}$ respectively denote the values of $\bar n$, $\bar m$ computed in Steps~\ref{step:findbeta}-\ref{step:barm} of Algorithm~\ref{alg:Jxi} for $j'{\text{th}}$ column of $J$, and $\mathcal B$ is an index set such that its $i{\text{th}}$ element stores the largest index of $\hat{\mathcal S}\left(Q_i\right)$.
\end{corollary}	
\begin{proof}	
	Considering the structure of matrix $J$, Definition~\ref{def:nzs} and the fact that  $\min\left(\bar m_j, m\right) > \bar n_j \ge n \ge \vert \mathcal F \vert$, $\forall j$ by construction, we have that
	\begin{equation}
		\label{eq:nnzJi}
		\mathcal S \left( J_{\mathcal F_i}\right) = \left\{\mathcal F_i\right\} \cup  \left(\bar n_{\mathcal F_i}, \min\left(\bar m_{\mathcal F_i},m\right) \right].
	\end{equation}
	From~\eqref{eq:approxnnzQMs} we note that $\bigcup\limits_{j=1}^{i} \mathcal S \left(J_{\mathcal F_j}\right) = \hat{\mathcal S}\left(Q_i\right)$, and using~\eqref{eq:nnzJi} we can rewrite 
	\begin{align}
		\hat{\mathcal S}\left(Q_i\right)	 &= \bigcup\limits_{j=1}^{i} \mathcal S \left(J_{\mathcal F_j}\right)\notag\\
		&=\left\{\bigcup\limits_{j=1}^{i}\left\{\mathcal{F}_j\right\}\right\} \cup \left\{\bigcup\limits_{j=1}^{i}\left(\bar n_{\mathcal F_j}, \min\left(\bar m_{\mathcal F_j},m\right) \right]\right\}\notag,\\
		&=\left\{\bigcup\limits_{j=1}^{i}\left\{\mathcal{F}_j\right\}\right\} \cup \left(\bar n_{\mathcal{F}_1},  \mathcal{B}_{i}\right]\label{eq:pnzsQiBi},
	\end{align}
	because observing~\eqref{eq:spyJ}, $\mathcal{F}_{j+1}>\mathcal{F}_j$ implies $\bar n_{\mathcal F_j}\leq \bar n _{\mathcal F_{j+1}}$. From result~\eqref{eq:approxnnzQ},~\eqref{eq:nnzJi} and definition of set $\mathcal B$,
	\begin{equation}\label{eq:recsetB}
		\mathcal B_i = \max\left(\mathcal{B}_{i-1}, \min\left(\bar m_{\mathcal{F}_i}, m\right)\right),
	\end{equation}
	which on substitution in~\eqref{eq:pnzsQiBi} completes the proof.
\end{proof}
Note that from~\eqref{eq:nnzJi} and result~\eqref{eq:nnzQ0},
\begin{equation}\label{eq:nnzJF1}
	\mathcal S\left(Q_1\right) = \mathcal S\left(J_{\mathcal{F}_1}\right) = \left\{\mathcal{F}_1, \left(\bar n_{\mathcal{F}_1}, \min\left(\bar m_{\mathcal{F}_1}, m\right)\right]\right\}.
\end{equation}
By definition of set $\mathcal B$ we have $\mathcal B_1 = \min\left(\bar m_{\mathcal{F}_1}, m\right)$ from~\eqref{eq:nnzJF1}, and hence $\mathcal{B}_i$ can be determined $\forall i$ by using~\eqref{eq:recsetB}.
\begin{corollary}\label{corr:zerorowsJFQ} $Q(i, j) = 0~ \forall i\in [1, n]\setminus \mathcal{F}, \forall j \in[1, \vert\mathcal{F}\vert]$. Also, $\forall j'\in [1, \vert\mathcal{F}\vert]$, $Q(i, j)=0~ \forall j \in [1, j')$ such that $i=\mathcal{F}_{j'}$.
\end{corollary}	
\begin{proof}Let $Q''$ denote the submatrix formed from the first $n$ rows of $Q$. Since $\bar n _{{\mathcal F}_1} > n$, from Corollary~\ref{corr:nnzJFQ} we can write $\bigcup\limits_{j=1}^{i}\left\{\mathcal{F}_j\right\}=\hat{\mathcal{S}}\left(Q''_i\right)$. Thus, referring this relation and Definition~\ref{def:pnzs}, if an index is not in the set $\mathcal F$, the corresponding row of $Q''$ and hence $Q$ has no \nz element. The latter part is proved by~\eqref{eq:approxnnzQ} considering the facts that $J$ is diagonal and $\mathcal F_{i+1}> \mathcal{F}_i$.
\end{proof}	
From Corollaries~\ref{corr:nnzJFQ}~and~\ref{corr:zerorowsJFQ} we infer that the \nz structure of all the $\vert\mathcal F \vert$ columns of $Q$ can be stored using a scalar for $\bar n_{\mathcal{F}_1}$ and two integer vectors of dimension $\vert \mathcal F \vert$ containing the index sets $\mathcal{F}$ and $\mathcal{B}$, where\\ $\mathcal B _i = \max\left(\min\left(\bar m_i, m\right), \bar m_{i-1}\right)$. In order to only compute the \nzs of $R$, while constructing each of its column, we need to find and store a scalar $j'$ as shown in Corollary~\ref{corr:nnzR}. This is done by using the relations described in Theorem~\ref{thm:nnzQ}, Corollary~\ref{corr:nnzR} and~\eqref{eq:pnzsQiBi}. Specifically, when computing the $i{\text{th}}$ column of $R$ ($i>1$), $j'$ is found by counting the number of times $B_j<\bar n_{\mathcal F _i}$ for increasing values of $j\in\left(\hat{j},i\right)$ until the condition is not satisfied, where $\hat j$ denotes the value of $j'$ for the $(i-1){\text{th}}$ column of $R$.

\subsection{Recursive updates}\label{sec:recQR}
In the primal active-set method, a change in the active-set corresponds to an index inserted in or deleted from the set~$\mathcal{F}$. We exploit the uniqueness of thin QR factorization in order to update the structure indicating sets $\mathcal F$ and $\mathcal B$. When an index $t$ is inserted in the active-set of bounds, the last column of $Q$ and the $i{\text{th}}$ column of $R$ are deleted such that $t=\mathcal{F}_i$, and the QR factorization is updated by applying Given's rotations that triangularize $R$. In this case $\mathcal F$ is simply updated to $\mathcal F' = \mathcal F \setminus \{t\}$ and $\mathcal B$ is updated such that~\eqref{eq:recsetB} is satisfied after removing its $i{\text{th}}$ index.  Morover, using Corollary~\ref{corr:zerorowsJFQ}, the Given's rotations are not applied on the $t{\text{th}}$ row of $Q$ which is simply zeroed. On the other hand, when an index $t$ is removed from the active-set of bounds, $\mathcal F$ is updated to $\mathcal F \cup \{t\}$ such that $\mathcal F _{j+1}>\mathcal F_j$, $\forall j$. If $t$ is inserted in $\mathcal F$ in the $j{\text{th}}$ position, an index is inserted in the $j{\text{th}}$ position of $\mathcal B$ using~\eqref{eq:recsetB} and the elements with position greater than $j$ are updated to satisfy~\eqref{eq:recsetB}.  Since the sparse structure of the updated QR factors is known during recursive updates, using $\mathcal F$, $\mathcal B$ and Corollary~\ref{corr:zerorowsJFQ}, the \textit{flops} for applying Given's rotations on rows of $Q$  and matrix-vector multiplications in the Gram-Schmidt (re)orthogonalization procedure are performed only on \nz elements. This makes the QR update routines significantly faster as is reflected in the numerical results described in Section~\ref{sec:results}. 
\subsection{Advantages and limitations}
The predicted \nz structure of the Q-factor via~\eqref{eq:approxnnzQ} is exact if and only if the set relation~\eqref{eq:nnzQ1} holds with equality. For~\eqref{eq:nnzQ1} to hold with equality for $Q$, $Q^{\top}_j J_{\mathcal {F}_i}$ must be \nz for all pairs of indices $i$ and $j$ referring the CGS orthogonalization in~\eqref{eq:CGS} and moreover the summation of \nzs in the RHS of~\eqref{eq:CGS} must result in a \nz. Even though theoretically this may not be the case for the matrices that we consider, due to finite precision computations which disallow perfect orthogonality, and the use of MGS with potentially multiple orthogonalizations to compute columns of $Q$, the predicted \nz structure of columns of $Q$ via Corollary~\ref{corr:nnzJFQ} rarely contains indices of zero elements, i.e., numerically it is an accurate estimate and often the exact \nz structure. Referring to Corollary~\ref{corr:nnzR} and~\cite[Algorithm 2]{bvlsTN}, the same fact leads to the conclusion that if multiple orthogonalizations (for numerical robustness) are performed, in the worst case, the upper-triangular part of the $R$ factor may have no zero elements. Nevertheless, the initial sparsity in $R$ before reorthogonalization is still exploited in its construction but the worst-case fill-in makes it necessary to use $R$ as a \textit{dense} upper-triangular matrix when solving the triangular system by back-substitution to compute the solution of the underlying LS problem.

From Theorem~\ref{thm:nnzQ}, we observe that the predicted \nz structure of columns $Q_j, \forall j\ge i$, would contain at least the indices of \nz elements in the $i{\text{th}}$ column of $J_{\mathcal F}$. Hence, in case $N_{\mathrm{u}}<N_{\mathrm{p}}$, referring the analysis in Section~\ref{sec:sparsityThinQR}, the fill-in of $Q$ can be reduced by a re-ordering of the decision variable vector in~\eqref{eq:zvec} such that the columns of $J$ corresponding to the variables $u_{k+N_{\mathrm{u}}-1}$ are moved to become its last columns. Note that even though this re-ordering does not optimize the fill-in of $Q$, for which dedicated routines exist in literature (cf.\ \cite{colamd}), it still allows a relatively simple and a computationally effective implementation of recursive thin QR factorization for the matrix of interest through a straightforward extension of the methods described in Section~\ref{sec:recQR}.

In order to benefit computationally from the recursive updates, a full storage of the thin QR factors is required. This causes greater memory requirement beyond a certain large problem size where a sparse-storage scheme would need smaller memory considering that with conventional sparse linear algebra, one would only compute and store the R factor while always solving the LS problem from scratch instead. However, the latter approach could turn out to be computationally much more expensive. Using the techniques discussed in Sections~\ref{sec:sparsityThinQR} and~\ref{sec:recQR} with a sparse-storage scheme could address this limitation specific to large-scale problems for memory-constrained applications but it needs a much more intricate implementation with cumbersome indexing, that is beyond the scope of this paper. 
\section{Numerical results}\label{sec:results}
\subsection{Software framework}
In order to implement the (nonlinear) MPC controller based on formulation~\eqref{eq:nllsbox} or~\eqref{eq:ALMbvnlls}, one only needs the code for Algorithm~\ref{alg:bvnlls}. The inner BVLS solver of~\cite{bvlsTN} could be  replaced by another algorithm that exploits sparsity via the abstract operators, such as the gradient-projection algorithm of~\cite{nesterov} we mentioned earlier. Besides, routines that evaluate the model~\eqref{eq:model} and the Jacobian matrices, i.e., the model coefficients in~\eqref{eq:arxmimovec}, are required from the user in order to evaluate the residual and perform the linearization step (or alternatively finite-differences) in BVNLLS. Note that an optimized self-contained code for these routines can easily be generated or derived by using symbolic tools such as those of MATLAB or the excellent open-source software CasADi~\cite{casadi}. This signifies that except for the user-defined model and tuning parameters, the software does not need any code-generation, as for a given class of performance indices the code for Algorithms~\ref{alg:bvnlls}-\ref{alg:Jtxi} does not change with the  application. 

The user is only required to provide the MPC tuning parameters and a symbolic expression for the model~\eqref{eq:model}, which eases the deployment of the proposed MPC solution algorithm in embedded control hardware. 
\subsection{Computational performance}
The results presented in this section are based on a library-free C implementation of BVNLLS based on Algorithms~\ref{alg:Jxi} and~\ref{alg:Jtxi}, and the BVLS solver based on the recursive thin QR factorization routines discussed in Section~\ref{sec:sparseQR}. The reader is referred to~\cite[Section 5]{nmpcbvls} for details on simulation settings, tuning parameters, constraints, and benchmark solvers related to the following discussion on the example problem, which consists of NMPC applied to a CSTR (continuous stirred tank reactor). All the optimization problems in the simulations referred below were solved until convergence\footnote{the optimality and feasibility tolerances for all solvers were tuned to $10^{-6}$ and $10^{-8}$ respectively, in order to achieve the same quality of solution at convergence for a fair comparison.}, on a Macbook Pro 2013 equipped with 8GB RAM and 2.6 GHz Intel Core i5 processor. Figure~\ref{fig:performance} illustrates the specific simulation scenario for which the execution times of the solvers were compared for increasing values of the prediction horizon. Figure~\ref{fig:conviol} shows that the equality constraints were satisfied with sufficient accuracy by BVNLLS. Hence, as also demonstrated in~\cite[Section 5, Figure 2]{nmpcbvls}, all the solvers including the proposed one yield the same control performance.

\begin{figure}[t!]
	\centering
	\includegraphics[]{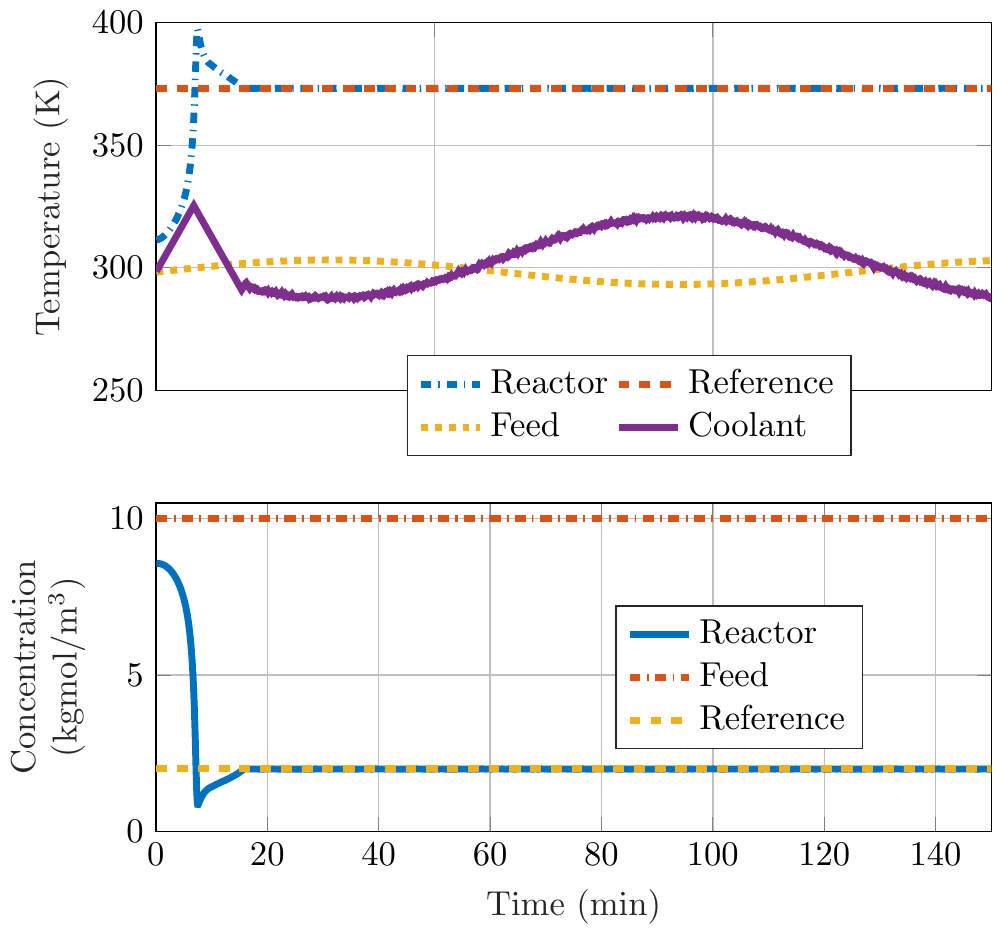}
	\caption{Closed-loop NMPC simulation trajectories of CSTR variables with $N_{\text{p}}=N_{\text{u}}=160$ (16 minutes).} 
	\label{fig:performance}
\end{figure} 
\begin{figure}[b!]
	\centering
	\includegraphics[]{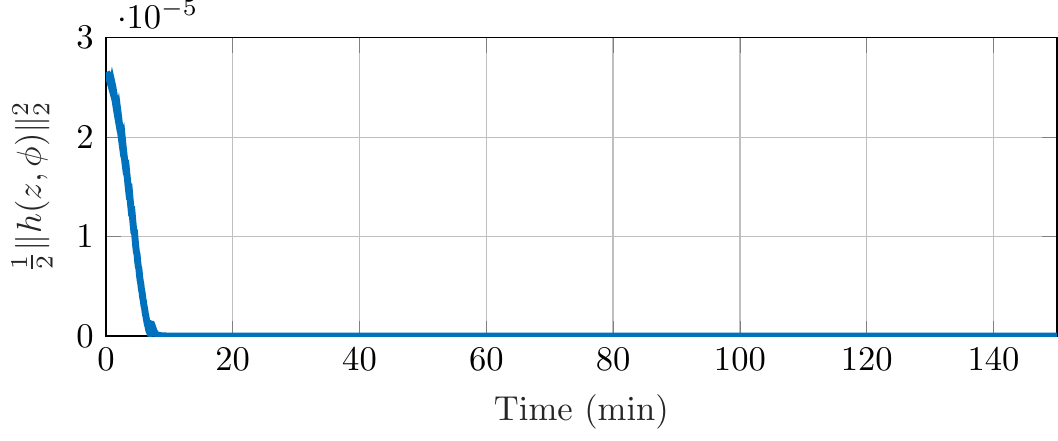}
	\caption{Worst-case equality constraint residuals ($N_{\text{p}}=N_{\text{u}}=160$) for BVNLLS with $\sqrt{\rho}=10^4$ during the simulation described by Figure~\ref{fig:performance}.} 
	\label{fig:conviol}
\end{figure}
Figure~\ref{fig:ctimeSmall} shows that the proposed methods (\emph{custom-sparse}) allow BVNLLS to outperform its \textit{dense} linear algebra based variant even on small-sized test problems by almost an order of magnitude on average. As compared to other solvers which are instead applied to the benchmark formulation~\eqref{eq:nlp}, i.e., the SQP solver (\texttt{fmincon}) of MATLAB and the interior-point solver (IPOPT) of~\cite{ipopt}, a reduction in computational time by around two orders of magnitude is observed for the small-sized test problems. This reduction can be credited to the fact that IPOPT, which is based on sparse \la routines, is more effective for large-sized problems, and that BVNLLS exploits warmstarts based on the previously computed solution which is provided from the second instance onwards. Note that the solvers fmincon/IPOPT were used because they are widely used as benchmarks (see for instance for IPOPT the recent paper~\cite{acados} and its references), and are available for reproducing the results, which also allows one to compare with another approach through the time ratios observed in the referred figures.

Figure~\ref{fig:ctimeLarge} suggests that despite being based on an active-set algorithm, the proposed sparsity-exploiting methods empower BVNLLS to significantly outperform the benchmarks even for large problems.
\begin{figure}[h!]
	\centering
	\includegraphics[]{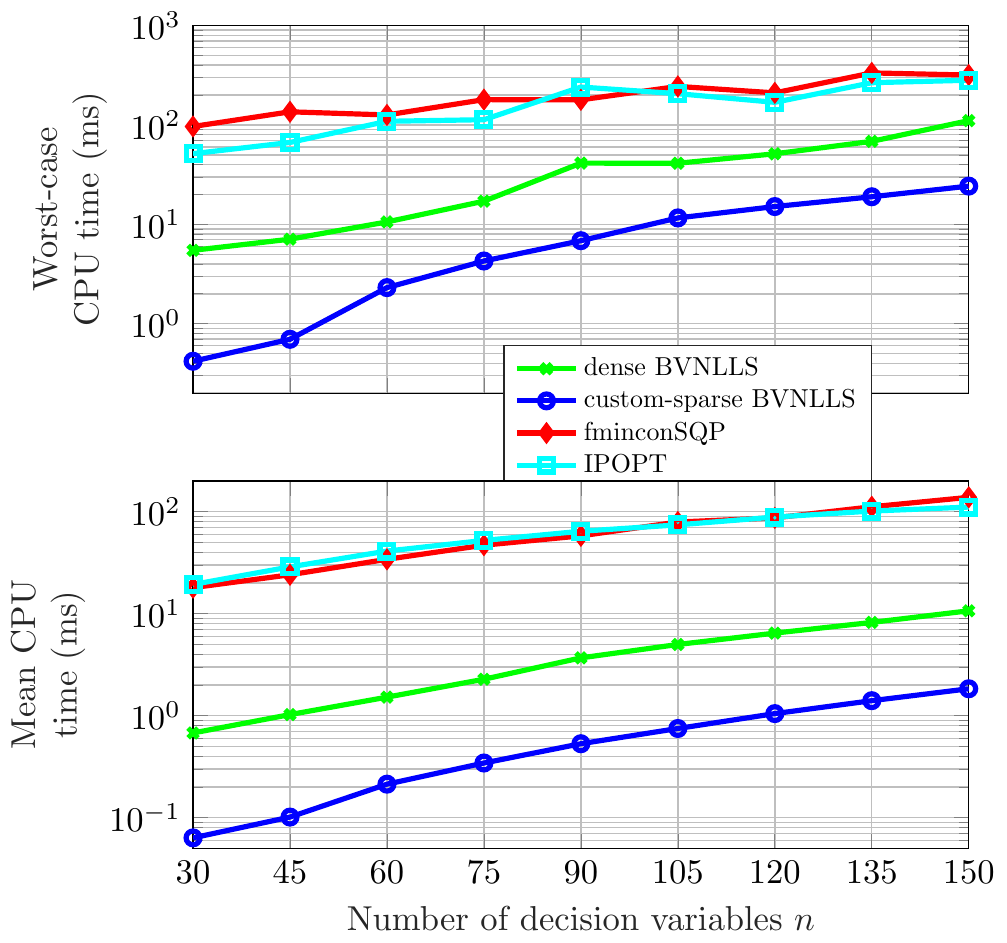}
	\caption{Computational time spent by each solver during NMPC simulation of CSTR~\cite{nmpcbvls} for increasing values of $N_{\mathrm{p}}=N_{\mathrm{u}}=n/3$, $n$ set of box-constraints and $2N_{\mathrm{p}}$ equality constraints.} 
	\label{fig:ctimeSmall}
\end{figure} 
\begin{figure}[h!]
	\centering
	\includegraphics[]{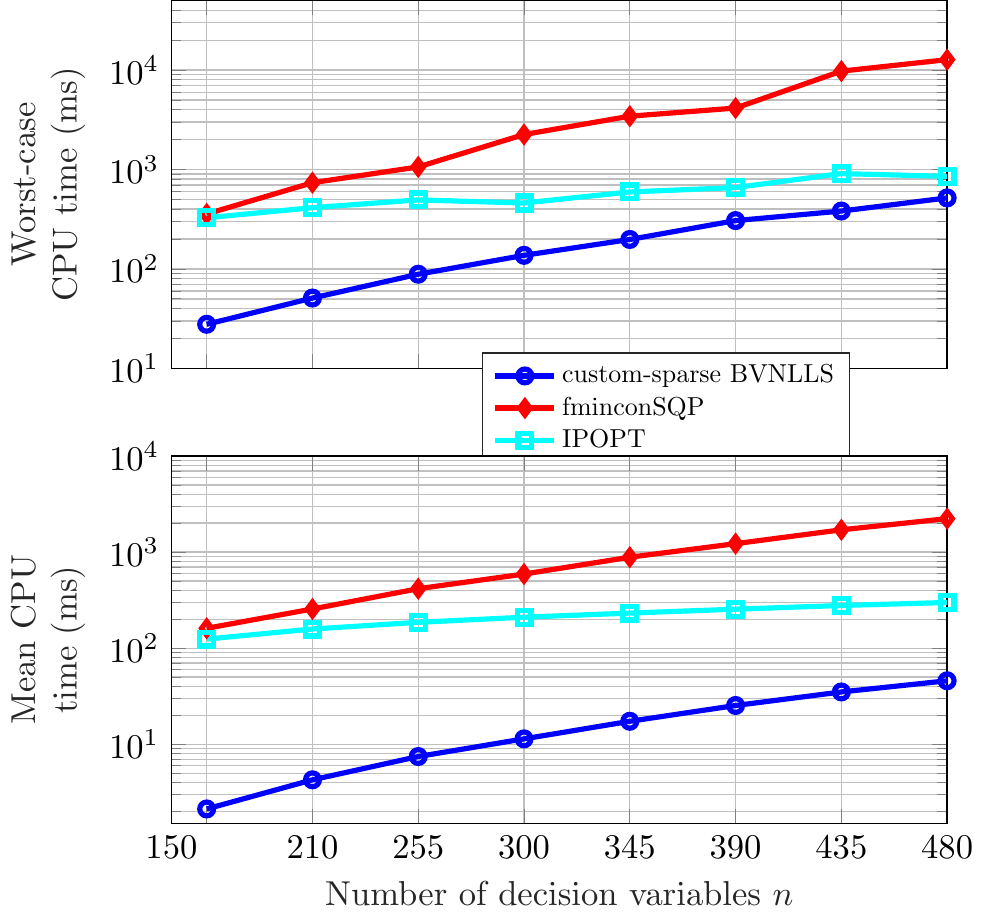}
	\caption{Computational time spent by each solver during NMPC simulation of CSTR for large values of $N_{\mathrm{p}}=N_{\mathrm{u}}=n/3$ with $n$ set of box-constraints and $2N_{\mathrm{p}}$ equality constraints.} 
	\label{fig:ctimeLarge}
\end{figure}

\section{Conclusions}
\label{sec:conclusions}
This paper has presented a new approach to solving constrained linear and nonlinear MPC problems that, by relaxing the equality constraints generated by the prediction model into quadratic penalties, allows the use of a very efficient bounded-variable
nonlinear least squares solver. The linear algebra behind the latter has been specialized in detail to take into account the particular structure of the MPC problem, so that the resulting required memory footprint and throughput are minimized for efficient real-time implementation, without the need of any external advanced linear algebra library. 

\section*{Acknowledgments}
The authors thank Laura Ferrarotti and Mario Zanon (IMT Lucca) and  Stefania Bellavia (University of Florence) for stimulating discussions concerning the convergence of BVNLLS.

\end{document}